\documentclass[aap,preprint]{imsart}

\RequirePackage{amsthm,amsmath,amsfonts,amssymb}
\RequirePackage[numbers]{natbib}
\RequirePackage[colorlinks,citecolor=blue,urlcolor=blue]{hyperref}
\RequirePackage{graphicx}

\usepackage{amsbsy}
\usepackage{mathrsfs,mathtools}

\usepackage{stmaryrd}

\usepackage{bm}

\usepackage[font=sf, labelfont={sf,bf}, margin=1cm]{caption}
\usepackage{epsfig}
\usepackage{latexsym}
\usepackage{xcolor}
\usepackage{ae,aecompl}
\usepackage{soul,framed}
\usepackage{comment}
\usepackage{dsfont}
\usepackage{amsthm}
\usepackage{stmaryrd}
\usepackage{smfhyperref}
\usepackage[capitalize]{cleveref}
\usepackage{pstricks}
\usepackage{enumerate}
\usepackage{tikz,animate,media9}						
\usepackage{pifont}
\usepackage{marvosym}
\usepackage{algorithm}
\usepackage{algorithmic}

\usepackage{bbm}

\usepackage[normalem]{ulem}

\startlocaldefs
\numberwithin{equation}{section}
\theoremstyle{plain}

\newtheorem{theorem}{Theorem}[section]

\newtheorem{corollary}[theorem]{Corollary}
\newtheorem{proposition}[theorem]{Proposition}
\newtheorem{lemma}[theorem]{Lemma}

\newtheorem{definition}[theorem]{Definition}
\theoremstyle{definition}
\newtheorem{remark}[theorem]{Remark}

 \definecolor{clou}{rgb}{0.8,0.25,0.5125}

\newcommand{\redu}{\ensuremath{\mathrm{red}}}

\newcommand{\convd}{\,{\buildrel (d) \over \longrightarrow}\,}
\newcommand{\convp}{\,{\buildrel (p) \over \longrightarrow}\,}
\newcommand{\Prb}[1]{\mathbb{P}\left(#1\right)}
\newcommand{\bxi}{\boldsymbol{\xi}}
\newcommand{\btxi}{\boldsymbol{\tilde{\xi}}}
\newcommand{\txi}{\tilde{\xi}}
\newcommand{\bzeta}{\boldsymbol{\zeta}}
\newcommand{\blambda}{\boldsymbol{\lambda}}
\newcommand{\btzeta}{\boldsymbol{\tilde{\zeta}}}

\newcommand{\btmu}{\boldsymbol{\tilde{\mu}}}
\newcommand{\tmu}{\tilde{\mu}}
\newcommand{\bmu}{\boldsymbol{\mu}}
\newcommand{\cT}{\mathcal{T}}
\newcommand{\cU}{\mathcal{U}}
\newcommand{\cW}{\mathcal{W}}
\newcommand{\N}{\mathbb{N}}
\newcommand{\Z}{\mathbb{Z}}
\newcommand{\R}{\mathbb{R}}
\newcommand{\bT}{\mathbb{T}}
\newcommand{\ctT}{\cT^{\flat}}
\newcommand{\be}{\mathbbm{e}}
\newcommand{\kT}{\cT^{\redu}}
\newcommand{\1}{\mathds{1}}
\newcommand{\bzero}{\boldsymbol{0}}
 \newcommand{\ctTl}{\cT^{\flat(\ell)}}
\newcommand{\mT}{\mathfrak{T}}

\newcommand{\Ex}[1]{\mathbb{E}[#1]}

\newcommand{\rvline}{\hspace*{-\arraycolsep}\vline\hspace*{-\arraycolsep}}

\newcommand{\BGW}{Bienaym\'e }
\newcommand{\bgw}{Bienaym\'e }

\endlocaldefs

\begin{document}

\begin{frontmatter}
\title{Scaling limits of  multitype Bienaym\'e trees}

\begin{aug}
\author[A]{\fnms{{Louigi}~\snm{Addario-Berry}\ead[label=e1]{louigi.addario@mcgill.ca}}},
\author[B]{\fnms{Philipp}~\snm{Beltran}\ead[label=e2]{philipp.beltran@tuwien.ac.at}},
\author[B]{\fnms{Benedikt}~\snm{Stufler}\ead[label=e3]{benedikt.stufler@tuwien.ac.at}}
\and
\author[C]{\fnms{Paul}~\snm{Th\'{e}venin}\ead[label=e4]{paul.thevenin@univ-angers.fr}}
\address[A]{Department of Mathematics and Statistics, McGill University\printead[presep={,\ }]{e1}}

\address[B]{Institute of Discrete Mathematics and Geometry, TU Wien\printead[presep={,\ }]{e2,e3}}
\address[C]{Facult\'{e} des Sciences, University of Angers\printead[presep={,\ }]{e4}}

\end{aug}

\begin{abstract}
We consider critical multitype Bienaymé trees that are either irreducible or possess a critical irreducible component with attached subcritical components. These trees are studied under two distinct conditioning frameworks: first, conditioning on the value of a linear combination of the numbers of vertices of given types; and second, conditioning on the precise number of vertices belonging to a selected subset of types. We prove that, under a finite exponential moment condition, the scaling limit as the tree size tends to infinity is given by the Brownian Continuum Random Tree. Additionally, we establish strong nonasymptotic tail bounds for the height of such trees. Our main tools include a flattening operation applied to multitype trees and sharp estimates regarding the structure of monotype trees with a given sequence of degrees.
\end{abstract}



\end{frontmatter}


\section{Introduction}

\subsection{Motivation}

\BGW processes (sometimes also called Galton--Watson processes) are random processes used to model the genealogy of a given population, in which individuals reproduce independently of each other with the same offspring distribution. Multitype \bgw processes are a natural generalization of \bgw processes, where each individual in the process is given a type and reproduces independently of the others in a way that only depends on its type. We refer to \cite{Har02} for an introduction to multitype processes. 

\bgw trees are a natural way of representing \bgw processes by trees, while keeping track of the genealogical structure. The root of the tree is the original ancestor of the population, and each nonroot individual is connected by an edge to its parent. The asymptotic properties of large \bgw trees conditioned on their size have been extensively studied, from a local point of view -- what does a finite neighbourhood of the root of the tree look like? -- as well as from a global point of view -- does the size-conditioned tree, viewed as a metric space, converge after rescaling distances?
Unlike in the monotype case, one of the tricky aspects of the multitype setting is that there exist different notions of size of a tree, which may lead to different behaviours.

In the monotype case, consider a probability distribution $\mu$ on $\N_0 := \{0, 1, \ldots \}$ and let $\cT_n$ be the \bgw tree in which each individual has offspring according to $\mu$, conditioned to have $n$ vertices (provided that this event holds with positive probability). If $\mu$ is critical (that is, has mean $1$) and has variance $\sigma^2\in (0,\infty)$, then $\cT_n$, viewed as a metric space where all edges have length $\sigma/(2\sqrt{n})$, converges in distribution to an (almost surely compact) random metric space $\cT_{\be}$ called the Brownian Continuum Random Tree (CRT) \cite{Ald91a,Ald91b,Ald93}. 

In the multitype case, the preceding result admits the following generalization. Write $[K]:= \{ 1, \ldots, K \}$,
let $\zeta^{(1)}, \ldots, \zeta^{(K)}$ be offspring distributions on $\bigcup_{n \geq 0}[K]^{n}$, and let $\cT^{(i)}$ be the associated multitype \bgw tree with $K$ types, where the root has type $i$ and where a vertex of type $j$ has offspring according to $\zeta^{(j)}$. Also, let $\cT_n^{(i)}$ be the tree $\cT^{(i)}$ conditioned to have $n$ vertices of type $i$, again provided that this event has positive probability.  
Under assumptions of criticality and finite variance, the tree $\cT_n^{(i)}$ converges after rescaling to the Brownian CRT; see \cite{HS21, Mie08}. In \cite{multitypeLevy} the authors consider the convergence to a multitype L\'evy tree when keeping track of the types.
In the former, the size of the tree is its number of root-type vertices, and the authors show that the study of $\cT_n^{(i)}$ boils down to investigating the structure of a monotype \bgw tree obtained from $\cT_n^{(i)}$ by keeping only the root-type vertices.

In this paper we consider a more general notion of size. Consider $\blambda := (\lambda_1, \ldots, \lambda_K) \in \N_0^K \setminus \{(0, \ldots, 0)\}$ nonnegative integers, and let $\cT^{(i)}_{\blambda,n}$ be the tree $\cT^{(i)}$ conditioned on
\begin{equation*}
	\sum_{j=1}^K \lambda_j \#_j (\cT^{(i)})  = n,
\end{equation*}
where $\#_j(T)$ denotes the number of vertices of type $j$ in a tree $T$. This type of size-conditioned tree was considered by Stephenson \cite{Ste18}, who described the annealed local limit of $\cT^{(i)}_{\blambda,n}$.

The main goal of this paper is to prove that, under assumptions of criticality, irreducibility and finite exponential moments, the size-conditioned tree $\cT^{(i)}_{\blambda,n}$ converges after rescaling to the Brownian CRT. The main idea is again to consider the reduced tree induced by the root-type vertices. Although this tree is no longer distributed as a size-conditioned \bgw tree, it is a mixture of random trees called {\em trees with prescribed degree sequences}, which allows us to prove its convergence to $\cT_{\be}$ using the main result of~\cite{BM14}. The second step is to bound the distance between a stretched version of the reduced tree and the original tree. This comparison also yields  tail bounds for the height of $\cT^{(i)}_{\blambda,n}$. Such tail bounds were proven for trees with given degree sequences in the recent work~\cite{MR4815972} (see also~\cite{zbMATH07596259,MR3077536,zbMATH06734667}), and we build upon them to establish similar bounds for the reduced tree and ultimately for $\cT^{(i)}_{\blambda,n}$.

We also provide extensions of our main theorems to reducible multitype \BGW trees with additional subcritical components.
  Such trees have close connections to a variety of combinatorial models, see also~\cite{zbMATH00975597,zbMATH05050594,zbMATH06526370}. In particular, limits for the bivariate case with specific conditionings have been studied in \cite{CS23,zbMATH06918043} as tools for the study of random graphs from restricted classes. We emphasize that, in general, reducible trees also admit limits different from the Brownian CRT or other stable trees, as recently demonstrated in~\cite{zbMATH08054723}.
Formal statements of our results appear in Sections~\ref{sec:main_irred} and~\ref{sec:main_red}, below.

\subsection{Background on random trees}

\subsubsection{Plane trees}

We use Neveu's formalism \cite{Nev86}. Let $\N:=\{1,2,\dots\}$ and define $\cU := \bigcup_{k \geq 0} \N^k$, the set of finite sequences of integers (with the convention that $\N^0=\{\varnothing\}$). For any $k \in \N$, an element $u$ of $\N^k$ is written $u=u_1 \cdots u_k$, with $u_1, \ldots, u_k \in \N$. For any $k \in \N$, $u=u_1\cdots u_k \in \N^k$, and $i \in \N$, we denote by $ui$ the element $u_1 \cdots u_ki \in \N^{k+1}$.

A plane tree $t$ is a subset of $\cU$ satisfying the following conditions:
(i) $\varnothing \in t$ (the tree has a root);
(ii) if $u=u_1\cdots u_n \in t$, then, for all $k \leq n$, $u_1\cdots u_k \in t$ (these elements are referred to as {\em ancestors} of $u$);
(iii) for any $u \in t$, there exists a nonnegative integer $k_u(t)$ such that, for every $i \in \N$, $ui \in t$ if and only if $1 \leq i \leq k_u(t)$ ($k_u(t)$ is called the {\em number of children} of $u$, or the {\em outdegree} of $u$, and $ui$ is said to be a child of $u$). The elements of $t$ are referred to as {\em vertices} of $t$, and vertices of $t$ without children are called {\em leaves} of $t$.  

 The {\em height} of a vertex is its length as a sequence of integers, and we denote by $|t|$ the total number of vertices of $t$. We denote by $(v_1(t), \ldots, v_{|t|}(t))$ the sequence of vertices of $t$ in lexicographical order.

We also define a partial order $\preceq$ on the set of vertices of $t$ by saying that $u \preceq v$ if $u$ is an ancestor of $v$. In particular, $\varnothing \preceq u$ for all $u \in t$. The last common ancestor $w$ of two vertices $u$ and $v$ is the largest (in lexicographical order) vertex being both an ancestor of $u$ and an ancestor of $v$. For vertices $x,y \in t$, we write $\llbracket x,y\rrbracket$ (resp.\ $\rrbracket x,y\llbracket$) for the set of vertices of the unique path in $t$ with endpoints $x$ and $y$, including (resp.\ excluding) $x$ and~$y$.

A plane tree $t$ can be represented as a graph-theoretic tree, where each nonroot vertex is connected to its parent by an edge. We will usually consider $t$ as a metric space whose points are the vertices of $t$, with its usual graph distance $d_t$. For $a>0$, we also denote by $at$ the tree $t$ viewed as a metric space, for the renormalized graph distance $a \cdot d_t$. Furthermore, we endow any finite tree $t$ with a mass measure $m_t$ defined as 
\begin{equation*}\label{def:massmeasure}
	m_t := \frac{1}{|t|} \sum_{v \in t} \delta_v.
\end{equation*}

Finally, we denote the set of finite plane trees by $\bT$.

\subsubsection{Multitype plane trees}

Fix $K \in \N := \left\{ 1, 2, \ldots \right\}$ and let $[K]$ be the set of types. A $K$-type plane tree is a pair $T := (t,e_t)$ where $t \in \bT$ is a plane tree and $e_t: t \to [K]$. For $u \in t$, $e_t(u)$ is called the type of the vertex $u$. We also denote by $\#_i(T)$ the number of vertices $u$ of the tree $t$ of type $i$, that is, such that $e_t(u)=i$. In particular, $|T| := \sum_{i=1}^K \#_i(T)$. Finally, we call $t$ the shape of the multitype tree $T$. For all $K \geq 1$, all $i \in [K]$, we let $\bT^{(K)}$ be the set of $K$-type plane trees and by $\bT^{(K,i)}$ the subset of $\bT^{(K)}$ of trees whose root has label $e_t(\varnothing)=i$. 
The degree sequence of a multitype tree $T$ is the unordered multiset  $\{(\ell_1, k_1), \ldots, (\ell_{|T|},k_{|T|})\}$ of elements of $[K] \times \N_0^K$, where, for all $i \in[|T|]$ and all $j \in [K]$, $\ell_i$ is the type of the vertex $v_i(T)$ and $k_i(j)$ is the number of children of type $j$ of the vertex $v_i$.

Given a sequence $\blambda := (\lambda_1, \ldots, \lambda_K) \in \N_0^K$, we set $\#_{\blambda}(T) := \sum_{i=1}^K \lambda_i \#_i(T)$, the linear combination of the number of vertices of each type.

\begin{remark}
	For convenience, in the rest of the paper, we consider here only trees whose root type is $1$, and will thus usually drop the superscript $1$, writing $T$ for $T^{(1)}$. However, the result still holds for any root type, and even for a random root type; see the discussion after Theorem \ref{thm:mainthm}.
\end{remark}

\subsubsection{Multitype \bgw trees}\label{def:mgwzeta}

We next define $K$-type \bgw trees, which are our main object of interest. For $K \in \N$, consider the set $\cW_K := \bigcup_{n \geq 0} [K]^n$ and let $\bzeta := (\zeta^{(i)})_{i \in [K]}$ be a family of probability distributions on $\cW_K$. For $i \in [K]$, we define a probability distribution on $\bT^{(K,i)}$ as follows. Consider a family of independent variables $(X_u^i, u \in \cU, i \in [K])$ with values in $\cW_K$, such that for all $(u,i) \in \cU \times [K]$, $X_u^i$ has law $\zeta^{(i)}$. We recursively construct a random $K$-type tree $\cT^{(i)} := (t, e_t)$ as follows:
$\varnothing \in t, e_t(\varnothing) = i$; if $u \in t$ and $e_t(u) = j$, then, for $k \in \N$, $uk \in t$ if and only if $1\leq k \leq |X_u^j|$ and in this case $e_t(uk)=X_u^j(k)$ (the $k$'th element of the sequence $X_u^j$).

In other words, the root of $\cT^{(i)}$ has type $i$, and vertices of type $j$ in $\cT^{(i)}$ have children according to $\zeta^{(j)}$, independently of each other and of all other vertices. We refer to $\cT^{(i)}$ as a $\bzeta$-\bgw tree with root type $i$.

\subsubsection{Projection}

We define the \textbf{projection} of a multitype probability distribution, which is later needed for flattening and reduction operations on trees.

\begin{definition}\label{def:projection}
	A $K$-type offspring distribution is a family $\bzeta:=(\zeta^{(1)}, \ldots, \zeta^{(K)})$  of probability distributions on $\cW_K$. For any $w \in \cW_K$ and any $j \in [K]$, denote by $\#_j w$ the number of $j$'s appearing in $w$. 
	We define the projection $\bmu:=(\mu^{(1)}, \ldots, \mu^{(K)})$ of $\bzeta$ as the family of probability distributions on $\N_0^K$ satisfying: for all $i \in [K]$, all $k_1, \ldots,k_K \geq 0$:
	\begin{equation*}
		\mu^{(i)}(k_1, \ldots, k_K) = \sum_{\substack{w \in \cW_K \\ \#_1 w=k_1, \ldots, \#_K w=k_K}} \zeta^{(i)}(w).
	\end{equation*}
\end{definition}

\subsubsection{The Brownian CRT}

We recall here the definition of the Brownian Continuum Random Tree (CRT), introduced by David Aldous \cite{Ald91a,Ald91b,Ald93}. Let $(\be_t)_{t \in [0,1]}$ be a standard Brownian excursion. We define a pseudo-distance on $[0,1]$ by saying that, for all $s \leq t$,
\begin{align*}
	d_{\be}(s,t) = \be_s+\be_t - 2 \min_{u \in [s,t]} \be_u,
\end{align*}
and $d_{\be}(t,s)=d_{\be}(s,t)$. We also define an equivalence relation $\sim_{\be}$ on $[0,1]$ as follows: for all $s\leq t \in [0,1]$, $s \sim_{\be} t$ if and only if $\be_s = \be_t = \min_{u \in [s,t]} \be_u$ , that is, $d_{\be}(s,t)=0$. The Brownian CRT is the metric space $\cT_{\be} := [0,1] / \sim_{\be}$, endowed with the projection $d_{\mathcal{T}_{\be}}$ of $d_{\be}$ onto $\cT_{\be}$ (observe that this is a distance on $\cT_{\be}$). See Figure~\ref{fi:crt} for a simulation. We root $\cT_{\be}$ at the equivalence class $\rho_{\be}$ of $0$.

We remark that $\cT_{\be}$ is also naturally endowed with a mass measure $m_{\mathcal{T}_{\be}}$ defined as the projection onto $\cT_{\be}$ of the Lebesgue measure on $[0,1]$.

It turns out that $(\cT_{\be}, \rho_{\be}, d_{\mathcal{T}_{\be}}, m_{\mathcal{T}_{\be}})$ is a universal scaling limit, in the sense that numerous models of trees have been shown to converge to it after rescaling of the distance, see e.g. \cite{BM14, Ald91a, CHK15}.

\subsubsection{Topology}

We endow the set of measured isometry-equivalence classes of marked compact measured metric spaces with the rooted Gromov--Hausdorff--Prokhorov distance $d_{GHP}$; see \cite{Vil09}, Section $27$, for an introduction. We denote by $d_H$ the Hausdorff distance between subsets of a given metric space.

\subsection{Main result for irreducible processes} \label{sec:main_irred}

A $K$-type offspring distribution is a family $\bzeta:=(\zeta^{(1)}, \ldots, \zeta^{(K)})$  of probability distributions on $\cW_K$. Recall that for any $w \in \cW_K$ and any $j \in [K]$, we denote by $\#_j w$ the number of $j$'s appearing in $w$. We assume that for at least one $1 \le i \le K$ we have
\begin{align}
	\zeta^{(i)}(\emptyset)>0,
\end{align}
where $\emptyset$ denotes the empty word, and that there exists $1 \le j \le K$ such that the probability for a type $j$ vertex to have at least two children is positive.

Finally, we denote by $\cT_n$ a $\bzeta$-\bgw tree $\cT$ with root type $1$ conditioned on the event that $\#_{\blambda} \cT=n$ (provided that this holds with positive probability), and we denote by $\varnothing$ the root of $\cT_n$.
In the paper, we will always consider distributions $\bzeta$ with finite exponential moments, that is:
\begin{equation}
	\exists z>1, \forall i \in [K], \sum_{w \in \cW_K} \zeta^{(i)}(w) z^{|w|} < \infty,
\end{equation}
where $|w|$ denotes the length of the word $w$. In addition to this, define a $K \times K$ matrix $M$ by: for all $i,j \in [K]$,
\begin{align*}
    M_{i,j}=\mathbb{E}[\#_j w^{(i)}],
\end{align*}
where $w^{(i)}$ is the random variable on $\cW_K$ with distribution $\zeta^{(i)}$. We always assume criticality, which means that the spectral radius of $M$ is equal to $1$ and non-degeneracy, that is, there exists $i \in [K]$ such that $\mu^{(i)}\left( \{(x_1,\ldots,x_K), \sum_{j=1}^K x_j \geq 2 \}\right) > 0$. Furthermore, we assume throughout Sections \ref{sec:cvreduced}, \ref{sec:cvflattened} and \ref{sec:cvoriginal} that $M$ is irreducible, meaning that for all $i,j\in[K]$ there exists $p\in \mathbb{N}$ such that $M^p_{i,j}\neq 0$. Our first theorem describes the scaling limit of $\cT_n$, under the above assumptions.

\begin{figure}[t]
	\centering
	\includegraphics{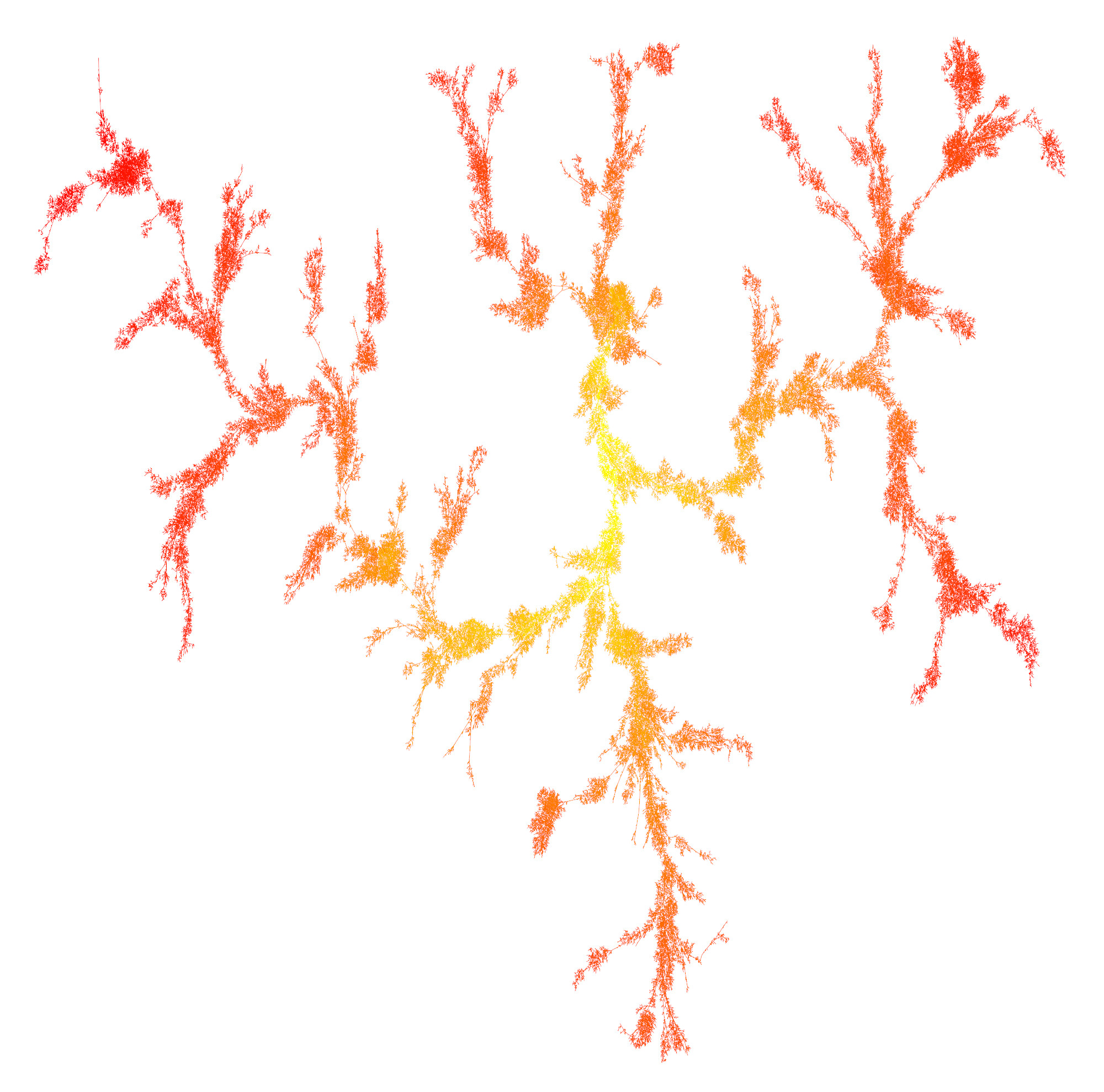}
		\centering
		\caption{The Brownian tree, approximated by a Bienaym\'e tree conditioned on having  one million vertices. Vertices are coloured on a gradient from yellow to red according to their closeness centrality.}
		\label{fi:crt}
\end{figure}

\begin{theorem}
	\label{thm:mainthm}
	Suppose that $\bzeta$  satisfies criticality, irreducibility and has finite exponential moments.
	Let $\blambda := (\lambda_1,\ldots, \lambda_K) \in \N_0^K \setminus \{(0, \ldots, 0) \}$. Then, there exists a constant $c_{\mathrm{scal}}$ depending only on $\blambda$ and on the distribution of $\bzeta$ such that, for the rooted Gromov--Hausdorff--Prokhorov topology, we have
	\begin{equation*}
		(\cT_n,  \frac{c_{\mathrm{scal}}}{\sqrt{n}} d_{\cT_n}, m_{\cT_n}) \convd (\cT_{\be},  d_{\mathcal{T}_{\be}}, m_{\mathcal{T}_{\be}}),
	\end{equation*}
	as $n \to \infty$ along integers satisfying $\Prb{\#_{\blambda}(\cT)=n}>0$. 
\end{theorem}

We note that $\Prb{\#_{\blambda}(\cT)=n}>0$ for infinitely many values of $n$; see, e.g. \cite{Ste18}, Proposition 2.2. The constant $c_{\mathrm{scal}}$ appearing in Theorem \ref{thm:mainthm} is explicit and given by

\begin{align*}
    c_{\mathrm{scal}}:=\frac{\sigma}{2} \sqrt{\sum_{i=1}^K \lambda_i a_i}.
\end{align*}

Here, $\mathbf{a}:=(a_1,\ldots,a_K)$ is the unique $1$-left eigenvector with positive coordinates of the mean matrix $M$ such that $\sum_{i=1}^K a_i=1$ (existence and uniqueness of $\mathbf{a}$ are guaranteed by the Perron-Frobenius theorem). The factor $\sigma>0$ is a variance term, satisfying
\begin{align}
\label{eq:sigma}
    \sigma^2=\sum_{i,j,k=1}^K a_i b_j b_k Q^{(i)}_{j,k},
\end{align}
with $\mathbf{b}:=(b_1,\ldots,b_K)$ the unique $1$-right eigenvector of $M$ with positive coordinates such that $\sum_{i=1}^K a_i b_i=1$, and
\begin{align*}
    Q^{(i)}_{j,j} = \sum_{\mathbf{z} \in \mathbb{N}_0^K} \mu^{(i)}(\mathbf{z}) z_j(z_j-1), \text{ and } Q^{(i)}_{j,k}=\sum_{\mathbf{z} \in \mathbb{N}_0^K}\mu^{(i)}(\mathbf{z}) z_j z_k \text{ for }j \neq k,
\end{align*}
where $\mu$ is the projection of $\bzeta$.
Again, existence and uniqueness of $\mathbf{b}$ are guaranteed by the Perron-Frobenius theorem. 

We remark that the scaling constant $c_{\mathrm{scal}}>0$ does not depend on the root type of $\cT$ and therefore Theorem~\ref{thm:mainthm} still holds if the root type of $\cT$ is chosen at random in an independent first step. 

Now let $\mathrm{H}(\cdot)$ denote the height of a tree, that is, the maximal graph distance from any vertex to the root vertex (equivalently, the maximal height of a vertex in the tree).
Our second theorem provides strong tail bounds for the rescaled height of $\cT_n$.

\begin{theorem}
	\label{te:bound}
	Suppose that $\bzeta$  satisfies criticality, irreducibility and has finite exponential moments. Let $\blambda := (\lambda_1,\ldots, \lambda_K) \in \N_0^K \setminus \{(0, \ldots, 0 )\}$. 
	Then there are constants $C,c>0$ depending only on $\blambda$ and on the distribution of $\bzeta$  such that
	\begin{align}
		\label{eq:tebound}
		\Prb{\mathrm{H}(\cT_n) > x } \le C \exp(-cx^2 /n ) + C\exp(-cx)
	\end{align}
	 for all $x>0$ and all integers $n$ with $\Prb{\#_{\blambda}(\cT)=n}>0$.
\end{theorem}

If, by construction, $\mathrm{H}(\cT_n) = O(n)$, then the bound simplifies to 
\begin{align}
			\Prb{\mathrm{H}(\cT_n) > x } \le C' \exp(-c'x^2 /n )
\end{align}
for some constants $C',c'>0$ that do not depend on $n$. This is the case, for example, when $\lambda_1, \ldots, \lambda_K \ne 0$, because then $\mathrm{H}(\cT_n) \le |\cT_n| \le \#_{\bm \lambda} \cT_n = n$. However, for general $\blambda$, the term $Ce^{-cx}$ in the bound is necessary.

The scaling limit of $\cT_n$ in Theorem~\ref{thm:mainthm} yields convergence in distribution of rescaled tree parameters such as the height, diameter, total path length,  Wiener index and many others.
The height bound in Theorem~\ref{te:bound} ensures convergence of all moments of these parameters. 

For example, the Wiener index (studied in the monotype case in~\cite{zbMATH01958290}) of $\cT_n$ is defined by
\[
	W(\cT_n) = \frac{1}{2} \sum_{u,v \in \cT_n} d_{\cT_n}(u,v).
\]
With $u_n$ and $v_n$ denoting independent samples of $m_{\cT_n}$, this may be expressed as a conditional expectation
\[
	W(\cT_n)  = \frac{|\cT_n|^2 }{2} \Ex{d_{\cT_n}(u_n, v_n) \mid \cT_n}.
\]
By Proposition~\ref{prop:concentration} (v) we know that $|\cT_n| /n \convp s_{\bm{\lambda}}$ for $s_{\bm{\lambda}} = c_1^{-1}\sum_{i=1}^K \Ex{\tilde{\xi}_i}$. By Theorem~\ref{te:bound} it follows that with $u,v$ denoting independent samples of $m_{\mathcal{T}_{\be}}$ 
\begin{align*}
	\frac{2 c_{\mathrm{scal}}}{s_{\bm{\lambda}}^2n^{5/2}} W(\cT_n) \convd \Ex{d_{\cT_{\be}}(u,v) \mid \cT_{\be}} 
\end{align*}
with
\begin{align*}
	\Ex{d_{\cT_{\be}}(u,v) \mid \cT_{\be}} 
	{\buildrel d \over =} \int_{0}^1 \int_{0}^1 (\be(x) + \be(y) - 2 \min_{\min(x,y) \le t  \le \max(x,y)}\be(t) ) \,\mathrm{d}x \,\mathrm{d}y.
\end{align*}
From Theorem~\ref{te:bound} it follows that $n^{-5/2} W(\cT_n)$ is arbitrarily highly uniformly integrable, yielding for all $p \ge 1$ that the $p$th moments of the rescaled Wiener index converge as well.

\subsection{Main result for reducible processes}\label{sec:main_red}

We turn to the case of reducible \BGW trees. One motivation for this is that in some scenarios we want to condition $\cT$ on having a given number of leaves. We could model this by modifying the offspring distribution so that there is an additional unique infertile type and then condition the whole random tree on having $n$ vertices of this type. However, the resulting offspring distribution is not irreducible. Therefore, we provide a second theorem that applies to branching processes of this kind.

Our setting is as follows. We consider a branching process with $K + K'$ types, where $K, K' \ge 1$ are integers, so that the mean matrix of the offspring distribution $\bm{\zeta}$ has the form
\begin{align}
	\label{eq:meanmatrixred}
	 A = 
\begin{pmatrix}
	M & \rvline & S \\
	\hline
	0 & \rvline & M'
\end{pmatrix}
\end{align}
with $M$ a $K \times K$ matrix, $M'$ a $K' \times K'$ matrix, and $S$ a $K \times K'$ matrix. We assume that $M$ is irreducible, meaning that for all $i,j \in [K]$ there exists $p \in \mathbb{N}$ such that the coefficient in the $i$th row and $j$th column of $M^p$ is nonzero. We assume that $M$ is critical, meaning its spectral radius is equal to $1$. We  assume that $M'$ is subcritical, that is, its spectral radius is strictly smaller than $1$. In order to exclude ``unused'' types, we assume without loss of generality that for each type $K+1 \le j \le K+K'$ there exists $p \in \mathbb{N}$ and $i \in [K]$ such that the coefficient in the $i$th row and $j$'th column of $A^p$ is positive. 

Our next condition is that $\bm{\zeta}$ has finite exponential moments. In order to avoid degenerate cases, we assume that there exists at least one type $i \in [K]$ that with positive probability has no offspring whose type belongs to $[K]$. We also assume that for at least one $i \in [K]$, the probability for a type-$i$ vertex to have more than one child is positive. 

Analogously to before, given nonnegative integers $\lambda_1, \ldots, \lambda_{K+K'}$ that are not all equal to zero we let $\cT_n$ denote a $\bm{\zeta}$-\BGW tree $\cT$ with root type $1$ conditioned on the event $\#_{\bm \lambda} \cT= n$.

\begin{theorem}
	\label{thm:reducible}
	Under the stated conditions, there exists a constant $c_{\mathrm{scal}}'>0$ such that
	\begin{equation*}
		(\cT_n, \frac{c_{\mathrm{scal}}'}{\sqrt{n}} d_{\cT_n}, m_{\cT_n}) \overset{(d)}{\longrightarrow} (\cT_{\be} , d_{\mathcal{T}_{\be}}, m_{\mathcal{T}_{\be}}),
	\end{equation*}
	as $n \to \infty$ along all positive integers $n$ satisfying $\Prb{\#_{\blambda}(\cT)=n}>0$. Moreover, there exist constants $C,c>0$ such that
	\[
		\Prb{\mathrm{H}(\cT_n)>x} \le C\exp(-cx^2 /n) + C\exp(-cx)
	\]
	uniformly for all $x>0$ and all integers $n$ with $\Prb{\#_{\blambda}(\cT)=n}>0$.
\end{theorem}
The scaling constant is given by 
\begin{align*}
	c_{\mathrm{scal}}':=\frac{\sigma}{2} \sqrt{\sum_{i=1}^{K+K'} \lambda_i a_i},
\end{align*} 
with $(a_1, \ldots, a_{K+K'})$ denoting the  left $1$-eigenvector of the mean matrix of $\bm{\zeta}$ that has nonnegative coordinates  and is normalized so that $\sum_{i=1}^K a_i = 1$. We emphasize that here we normalize so that the sum of only the first $K$ coordinates equals $1$.  This differs from the irreducible case, in which we normalized so that the sum of all coordinates equalled~$1$. 
Existence and uniqueness of this vector are verified in the proof of Theorem~\ref{thm:reducible}.

As before, the scaling factor does not depend on the root type, so Theorem~\ref{thm:reducible} still holds for any fixed or random root type from $[K]$ (but does not cover the case where the root has type larger than $K$).

We emphasize that we assume $M'$ to be subcritical and that this assumption is necessary. For example,  consider a $2$-type \BGW tree with type-$1$ vertices having a $\mathrm{Poisson}(1)$ number of type-$1$ children, and an independent $\mathrm{Poisson}(1)$ number of type-$2$ children. If type-$2$ vertices have a $\mathrm{Poisson}(1/2)$ number of type-$2$ children and nothing else, then Theorem~\ref{thm:reducible} applies and the scaling limit is the Brownian tree. If type-$2$ vertices have, instead, a $\mathrm{Poisson}(1)$ number of type-$2$ children and nothing else, then the scaling limit differs from the Brownian tree; see~\cite{zbMATH08054723}.

\subsection{Main result when conditioning by types}

Suppose that we are in the setting of the irreducible or the reducible case described above. One way to unify this is to assume the setting in the reducible case but allow $K'\ge 0$ instead of $K'>0$.

So far, we conditioned the \BGW tree $\cT$ on the event $\#_{\bm \lambda} \cT = n$. Instead we may fix a nonempty set $I$ of types, a sequence $x(n) =(x_{i}(n))_{i \in I} \in \mathbb{N}^I$, $n \ge 1$ and let $\mT_n$ denote the result of conditioning $\cT$ on 
\[
	(\#_i \cT)_{i \in I} = x(n).
\] 
Throughout, we only consider sequences $x(n)$ for which this event has positive probability.  
We will mainly be interested in the case where 
\begin{align}
	\label{eq:cccond}
	x_i(n) = a_i n + O(\sqrt{n})
\end{align}
for each $i \in I$ as $n \to \infty$, with, as before, $(a_1, \ldots, a_{K+K'})$ the unique left eigenvector with eigenvalue $1$ of the mean matrix in~\eqref{eq:meanmatrixred}  with nonnegative coefficients normalized so that $\sum_{i=1}^K a_i =1$, and we impose this as a condition in the next result.

\begin{theorem}
	\label{thm:bytype}
	Under the stated conditions, there exists a constant $\mathfrak{c}_{\mathrm{scal}} = \sigma /2$ such that
	\begin{equation*}
		(\mT_n,  \frac{\mathfrak{c}_{\mathrm{scal}}}{\sqrt{n}} d_{\mT_n}, m_{\mT_n}) \overset{(d)}{\longrightarrow} (\cT_{\be},  d_{\mathcal{T}_{\be}}, m_{\mathcal{T}_{\be}}),
       \end{equation*}
	as $n \to \infty$ along all positive integers $n$ satisfying $\Prb{(\#_i \cT)_{i \in I} = x(n)}>0$. Moreover, there exist constants $C,c>0$ such that
	\[
	\Prb{\mathrm{H}(\mT_n)>x} \le C\exp(-cx^2 /n) + C\exp(-cx)
	\]
	uniformly for all $x>0$ and all positive integers $n$ with $\Prb{(\#_i \cT)_{i \in I} = x(n)}>0$.
\end{theorem}

The scaling factor does not depend on the root type, hence Theorem~\ref{thm:bytype} still holds for any fixed or random root type from $[K]$. Note that this does not cover the case where the root has type larger than $K$.

Note that~\eqref{eq:cccond} requires $x(n)$ to roughly point in the direction of the corresponding restriction of $(a_1, \ldots, a_{K+K'})$.  In cases where it points in a  different direction it may be possible to switch to a suitable offspring distribution without affecting the law of the conditioned tree.

Specifically, Section 5 in~\cite{MR3476213} describes a tilting operation: Let
\begin{align*}
	\phi_i(z_1, \ldots, z_{K+K'}) &= \sum_{w \in \cW_{K+K'}} \zeta^{(i)}(w) \prod_{j=1}^{K+K'} z_j^{\#_j w}.
\end{align*}
Given a vector $\bm{\theta} = (\theta_1, \ldots, \theta_{K+K'})$ with positive coordinates 
and
\begin{align}
	\phi_i(\theta_1, \ldots, \theta_{K+K'}) < \infty 
\end{align}
for all $1 \le i \le K+ K'$, consider the offspring distribution $\bzeta_{\bm{\theta}}:=(\zeta^{(1)}_{\bm{\theta}}, \ldots, \zeta^{(K+K')}_{\bm{\theta}})$ with
\[
	\zeta_{\bm{\theta}}^{(i)} (w) = \frac{1}{\phi_i(\theta_1, \ldots, \theta_{K+K'})} \zeta^{(i)}(w) \prod_{j=1}^{K+K'} \theta_j^{\#_j w}, \qquad 1 \le i \le K+K'
\]
for all $w \in \cW_{K+K'}$. The mean matrix $	A^{\bm{\theta}} = \left( A_{i,j}^{\bm{\theta}} \right)_{1 \le i,j \le K+K'}$ of $\bzeta_{\bm{\theta}}$  has coefficients given by
\[
	A_{i,j}^{\bm{\theta}} = \frac{\theta_j \frac{\partial \phi_i}{\partial z_j}   (\theta_1, \ldots, \theta_{K+K'})}{ \phi_i(\theta_1, \ldots, \theta_{K+K'}) }.
\]

For $I = [K+K']$ the result of conditioning a Bienaym\'e tree with offspring distribution $\bzeta_{\bm{\theta}}$ on having precisely $x_i(n)$ many vertices of type $i$ for all $i \in [K+K']$ is identically distributed as $\mT_n$, see~\cite{MR3476213}, Theorem 5.1. Note that when $I$ is a proper subset of $[K+K']$ then this does not need to hold. 

Hence, we restrict ourselves to the case $I=[K+K']$. The question is, if~\eqref{eq:cccond} is not satisfied and instead we have
\[
                x_i(n) = \bar{a}_i n + O(\sqrt{n})
\]
for constants $\bar{a}_i \ge 0$, $i \in [K+K']$,  can we find parameters $\theta_1, \ldots, \theta_{K+K'}$ such that $\bzeta_{\bm{\theta}}$ satisfies the assumptions in Theorem~\ref{thm:bytype} and such that the left $1$-eigenvector $(a_i^{\bm{\theta}})_{1 \le i \le K+K'}$ of $A^{\bm{\theta}}$ with nonnegative coordinates normalized to $\sum_{i=1}^{K} a_i^{\bm{\theta}} = 1$ (whose existence and uniqueness is then guaranteed) satisfies $a_i^{\bm{\theta}} = \bar{a}_i$ for all $i \in I$?

For a given offspring distribution, it may be possible to directly compute a suitable tilting parameter $\bm{\theta}$. We provide an application in Subsection~\ref{sec:biconditioned}. There are also general results on the existence of tiltings in a more restricted setting, see Subsection~\ref{sec:tilting} below.

\subsubsection{Application to biconditioned monotype trees}
\label{sec:biconditioned}

Using Theorem~\ref{thm:bytype}, we recover and slightly extend a result of~\cite{zbMATH07768803} on biconditioned monotype trees: Consider a critical monotype Bienaym\'e tree $\mathfrak{t}$ whose offspring distribution $(q_i)_{i \ge 0}$ has finite exponential moments, and satisfies the usual constraints $q_0 >0$ and $q_0+q_1 <1$. We would like to study the asymptotic shape of the tree $\mathfrak{t}_n$ obtained by conditioning $\mathfrak{t}$ on having $n$ vertices and $k_n$ leaves, with
\begin{align}
	\label{eq:va1}
    k_n = \alpha n + O(\sqrt{n})
\end{align}
for some constant $0<\alpha<1$. We can model this as a special case of the tree $\mT_n$ with $K=1$ and $K'=1$. We choose the two-type offspring distribution $\bm{\zeta}=(\zeta^{(1)}, \zeta^{(2)})$ such that
\begin{align*}
\zeta^{(1)}(\emptyset) &= 0, & \zeta^{(1)}(w) &= q_{a+b} (1-q_0)^{a-1} q_0^{b}, \\
\zeta^{(2)}(\emptyset) &= 1, & \zeta^{(2)}(w) &= 0
\end{align*}
for $w \in \cW_2 \setminus \{\emptyset\}$ with $a:= \#_1 w$ and $b:= \#_2 w$. If we select the root type of $\mT$ according to an independent random choice with outcome $1$ with probability $1-q_0$, then up to the types of vertices the $\bm{\zeta}$-Bienaym\'e tree $\cT$ is distributed like $\mathfrak{t}$, and the result $\mT_n$ of conditioning $\cT$ on $(\#_1 \cT, \#_2 \cT) = (n- k_n, k_n)$ is distributed like $\mathfrak{t}_n$. Whenever $n>1$ the root type of $\mT_n$ is equal to $1$ almost surely, hence instead of a random root type we may set it deterministically to $1$ and only consider $n>1$.

For parameters $\bm{\theta} = (\theta_1, \theta_2)$ with $\theta_1, \theta_2>0$, the tilted offspring distribution has mean matrix 
\[
A = \frac{\phi_{\ge 1} '( (1-q_0)\theta_1 + q_0 \theta_2)}{\phi_{\ge 1}((1-q_0)\theta_1 + q_0 \theta_2)}
\begin{pmatrix}
	(1-q_0)\theta_1 &  q_0\theta_2 \\
	0 &  0
\end{pmatrix}
\]
with $\phi_{\ge1}(z) := (1-q_0)^{-1} \sum_{k \ge 1} q_k z^k$. Setting $t := \frac{q_0}{1-q_0} \frac{\theta_2}{\theta_1}$ and $\psi_{\ge1}(z) := \frac{z\phi_{\ge 1}'(z)}{\phi_{\ge 1}(z)}$, this simplifies to
\[
A =  \frac{1}{1+t} \psi_{\ge 1}( (1-q_0)\theta_1(1+t) )
\begin{pmatrix}
	1 &  t \\
	0 &  0
\end{pmatrix}.
\]
Let $\rho_\phi>1$ denote the radius of convergence of $\phi_{\ge 1}(z)$. For $z \in ]0,\rho_\phi[$, let $X$ denote the random positive integer with probability generating series $\Ex{v^X} = \phi_{\ge1}(vz) / \phi_{\ge 1}(z)$. Differentiating yields $\psi'_{\ge1}(z) = \frac{1}{z}\mathbb{V}(X)  >0$. Note that $\psi_{\ge 1}(0)$ is equal to the smallest positive integer $r_0 \ge 1$ with $q_{r_0} >0$.
Hence $\psi_{\ge 1}: [0, \rho_{\phi}[ \to [r_0, \infty[$ is continuous, strictly increasing, with $\psi_{\ge 1}(0) = r_0$, $\psi_{\ge 1}(1) >1$, and $\nu := \lim_{z \to \rho_\phi} \psi_{\ge 1}(z) \in ]1, \infty]$.

Now, if
\begin{align}
	\label{eq:va2}
	1 - 1/r_0 < \alpha < 1 - 1/\nu
\end{align}
with the convention $1/\infty  = 0$, then there exists $0< z_0 < \infty$ with $z_0 \le \rho_\phi$ such that $\psi_{\ge 1}(z_0) = 1 / (1-\alpha)$. Setting $\theta_1= z_0 (1-\alpha) / (1-q_0)$ and $\theta_2 = \alpha z_0 / q_0$ we obtain
\[
A = \begin{pmatrix}
	1 &  \alpha/(1-\alpha) \\
	0 &  0
\end{pmatrix}.
\]
Its left $1$-eigenvector normalized to have the sum of the first $K$ coordinates equal to $1$ is given by
\[
(a_1, a_2) = (1, \alpha/(1-\alpha)).
\]
It is collinear to $(1-\alpha, \alpha)$. The conditioning fulfills \\
$(\#_1 \cT, \#_2 \cT) = \left(n- k_n, \frac{\alpha}{1-\alpha}(n-k_n)+O(\sqrt{n})\right).$
Hence we may apply Theorem~\ref{thm:bytype} to the tilted offspring distribution $\bm{\zeta}_{(\theta_1, \theta_2)}$, yielding:

\begin{corollary}
	If~\eqref{eq:va1} and~\eqref{eq:va2} hold, then there exists a constant $\mathfrak{c}_{\mathrm{scal}}>0$ such that
	\begin{equation*}
		(\mathfrak{t}_n,  \frac{\mathfrak{c}_{\mathrm{scal}}}{\sqrt{n}} d_{\mathfrak{t}_n}, m_{\mathfrak{t}_n}) \overset{(d)}{\longrightarrow} (\cT_{\be},  d_{\mathcal{T}_{\be}}, m_{\mathcal{T}_{\be}}),
	\end{equation*}
	as $n \to \infty$ along all positive integers $n$ satisfying that the probability for $\mathfrak{t}$ to have $n$ vertices and $k_n$ leaves is positive. Moreover, there exist constants $C,c>0$ such that
	\[
	\Prb{\mathrm{H}(\mathfrak{t}_n)>x} \le C\exp(-cx^2 /n) + C\exp(-cx)
	\]
	uniformly for all $x>0$ and all positive integers $n$ satisfying that the probability for $\mathfrak{t}$ to have $n$ vertices and $k_n$ leaves is positive.
\end{corollary}
Note that here we made use of the fact that Theorem~\ref{thm:bytype} permits reducible offspring distributions. As noted above, the scaling limit is a slight extension of a result by~\cite{zbMATH07768803}, which assumed $k_n = \alpha n + O(1)$ instead of~\eqref{eq:va1}. We note that recent work~\cite{dan2025limitsbiconditionedbienaymegaltonwatsontrees} studies limits of this model when $k_n$ is constant.

\subsubsection{Existence of tiltings for entire irreducible offspring distributions}
\label{sec:tilting}

It is difficult to determine general conditions that ensure the existence of an appropriate tilting. We apply the results of \cite{PAT25}. Let us make the following assumptions on the offspring distribution $\bzeta$:
\begin{itemize}
    \item the mean matrix is irreducible, that is $K'=0$;
    \item $\bzeta$ is entire, that is, for all $i \in [K]$, all $z>0$:
    \begin{align*}
        \sum_{w \in \cW_K} \zeta^{(i)}(w) z^{|w|} < \infty;
    \end{align*}
    \item $\bzeta$ is nonlocalized: that is, for any $X \in \R^K \setminus \{\bm{0}\}$, there exists $i \in [K]$ such that, if $(k_1, \ldots, k_K)$ is distributed according to $\mu^{(i)}$, then $\sum_{j=1}^K X_j k_j$ is not deterministic.
    \item $\bzeta$ is finite: that is, for all $i\in [K]$, $\Prb{|\cT^{(i)}|<\infty}>0$.
\end{itemize}

Following \cite{PAT25}, Definition 6.1, we also define the notion of strongly accessible direction.

\begin{definition}
Let $X\in[0,+\infty[^K\setminus\{\bzero\}$. We say that $X$ is {\bf accessible} if there exists a tree $T$ with at least two vertices and with $\Prb{\cT=T}>0$, such that the following holds.
\begin{itemize}
    \item the root of $T$ has a type $i_0$ with $\zeta^{(i_0)}(\varnothing)>0$;
    \item at least one leaf of $T$ also has type $i_0$; 
    \item the vectors $(\tilde{N}_1(T),\dots,\tilde{N}_K(T))^{\intercal}$ and $X$ are collinear, where $\tilde{N}_j(T)$ is the number of nodes of type $j$ in $T$ (counting the leaves but excluding the root).
\end{itemize}
We say $X$ is \textbf{strongly accessible} if it belongs to the interior of the convex hull of the set of all accessible directions. In particular, all coordinates of $X$ are positive.
\end{definition}

In other words, accessible directions are the vectors of type proportions that can be realized by a finite tree with positive probability, while strongly accessible directions are vectors which are in the interior of the convex hull of accessible directions.


\begin{corollary}
\label{thm:bytype2}
Suppose that $\bzeta$ is entire, finite, irreducible, nondegenerate and nonlocalized. 
Let $Y \in ]0, +\infty[^K$ be strongly accessible, and $I=[K].$ 
Let  $\mT_n$ be the \BGW tree with offspring distribution $\bzeta$, conditioned on $(\#_i\cT)_{i \in I} = x(n)$ where, for all $i \in I$:
\begin{align*}
x_i(n)=Y_i n + O(\sqrt{n}).
\end{align*}
Then there exists a constant $\mathfrak{c}_{\mathrm{scal}}(\bzeta, Y)>0$ such that  
\begin{equation*}
		(\mT_n,  \frac{\mathfrak{c}_{\mathrm{scal}}(\bzeta,Y)}{\sqrt{n}} d_{\mT_n}, m_{\mT_n}) \overset{(d)}{\longrightarrow} (\cT_{\be},  d_{\mathcal{T}_{\be}}, m_{\mathcal{T}_{\be}}),
       \end{equation*}
	as $n \to \infty$ along all positive integers satisfying $\Prb{(\#_i\cT)_{i \in I} = x(n)}>0$. Moreover, there exist constants $C,c>0$ such that
	\[
	\Prb{\mathrm{H}(\mT_n)>x} \le C\exp(-cx^2 /n) + C\exp(-cx)
	\]
	uniformly for all $x>0$ and all positive integers $n$ satisfying $\Prb{(\#_i\cT)_{i \in I} = x(n)}>0$.
\end{corollary}
In particular, we do not require here that $\bzeta$ is critical.

The proof of Corollary \ref{thm:bytype2} is immediate from Theorem \ref{thm:bytype}. Indeed, the assumptions made allow us to directly apply \cite{PAT25}, Theorem 3.7, which ensures the existence of parameters $\bm{\theta}$ such that $\bzeta_{\bm{\theta}}$ is critical and its left Perron-Frobenius eigenvector is collinear with $Y$. As we assumed $\bm{\zeta}$ to be entire, it follows that $\bm{\zeta}_{\bm{\theta}}$ is regular critical. We then apply Theorem \ref{thm:bytype} to the tilted law $\bzeta_{\bm{\theta}}$. Note that the assumption that $\bm{\zeta}$ is entire is a fundamental assumption in the results of \cite{PAT25}. Suitable tiltings may exist in a broader context (for example also in the reducible setting considered above), but currently there is no general result that guarantees their existence.

\paragraph*{Outline of the paper}

We start by defining in Section \ref{sec:trees} operations on trees which we call {\em flattening} and {\em reduction}. We prove  convergence of the rescaled reduced tree in Section \ref{sec:cvreduced}, deduce from it  convergence of the flattened tree in Section \ref{sec:cvflattened}, and finally use it in Section \ref{sec:cvoriginal} to prove  convergence of the rescaled multitype size-conditioned \bgw tree to the Brownian CRT. The main tool of this last step is a blow-up operation, which can be thought of as an inverse (in distribution) of the reduction operation. In Section~\ref{sec:reducible} we prove Theorem~\ref{thm:reducible}, and in Section~\ref{sec:types} we prove Theorem~\ref{thm:bytype}. Finally, Section~\ref{sec:app} is an appendix containing some auxiliary technical results.

\paragraph*{Notation}
We use the conventions that $\N=\{1,2,\ldots\}$ and $\N_0=\{0,1,\ldots\}$.
Throughout, we denote by  $\convp$ and $\convd$ convergence in probability and distribution, respectively. 
For any two real-valued functions $f,g$, we say that $f=O(g)$ if  $\limsup\limits_{x\to \infty}|\frac{f(x)}{g(x)}|<\infty$ and that $f=\Theta(g)$ if $0<\liminf\limits_{x\to\infty}|\frac{f(x)}{g(x)}|\leq \limsup\limits_{x\to\infty}|\frac{f(x)}{g(x)}| < \infty.$ We also denote by $\mathbb{E}[X]$ and $\mathbb{V}[X]$ the respective expectation and variance of a random variable $X$. For all $n \geq 1$, we denote by $[n]$ the set $\{1,\ldots,n\}$. 
Since the sequence $\blambda$ will be fixed once and for all, and that all our trees will have root type $1$, we will drop the dependence in $\blambda$ and $i$ to ease the notation, and thus write $\cT_n$ for $\cT^{(i)}_{\blambda,n}$.

\section{Operations on random trees}
\label{sec:trees}

We define here two operations on trees, namely \emph{flattening} and \emph{reduction}.

\subsection{Flattening and reduction operations}
\label{ssec:operations}

Let $T \in \bT^{(K,1)}$ be a $K$-type tree with root type $1$. We define two operations on $T$, which we respectively call \textit{flattening} and \textit{reduction} of the tree $T$. Roughly speaking, flattening $T$ consists in reattaching all vertices - except the root - to their most recent ancestor of type $1$, so that all vertices of type different from $1$ become leaves. Reducing $T$ consists in removing from the flattened tree all vertices of type different from $1$. In particular, the flattened tree has as many vertices as $T$, while the reduced tree does not. A reduced tree is by definition a monotype tree (all its vertices have type $1$).

Let us now define these operations more formally. We start by defining blobs, which are roughly speaking made of all vertices in the original tree sharing the same most recent ancestor of type $1$.

\begin{definition}[Blob]
	Let $T$ be a multitype tree and $x \in T$ of type $1$. We define the blob $B_x$ associated to $x$ as
	\begin{align*}
		B_x := \bigcup_{a \in A_x} \llbracket x,a \rrbracket \cup \{x\},
	\end{align*}
	where $A_x := \{a \in T, a \neq x, x \preceq a, \forall y \in \rrbracket x,a \llbracket \, e_T(y) \neq 1 \}$.
\end{definition}

Note that the root of a blob is always of type $1$, and leaves of a blob that are not of type $1$ are also leaves in the original tree. The {\em reduced} tree $T^{\redu}$ is obtained by contracting all blobs and discarding the vertices of type different from $1$.

\begin{definition}[Reduced tree]\label{def:reducedtree}
	Let $T$ be a multitype tree with root type $1$. We define the reduced tree $T^{\redu}$, which is a monotype tree with $\#_1(T)$ vertices, as follows. Letting $x \in T$ be of type $1$, we set $A^{(1)}_x := \{u \in A_x \setminus \{x\}, e_T(u)=1 \}$. We denote the elements of $A^{(1)}_x$ by $c_1(x), \ldots, c_{|A^{(1)}_x|}(x)$ in lexicographical order. We recursively define a set $T^{\redu} \subseteq \cU$ and a map $\pi: T^{\redu} \rightarrow \cU$  as follows.
	\begin{itemize}
		\item $\varnothing \in T^{\redu}$ and $\pi(\varnothing)=\varnothing$;
		\item for all $u \in T^{\redu}$, all $k \geq 1$, $uk \in T^{\redu}$ if and only if $1 \le k \leq |A^{(1)}_{\pi(u)}|$; for all such $u,k$, we set $\pi(uk)=c_k(\pi(u))$.
	\end{itemize}
	Then, observe that $T^{\redu}$ is a tree, which we call the reduced tree.
\end{definition}

In particular, the planar structure of $T^{\redu}$ is naturally obtained from the one of $T$ as the only one preserving the relative lexicographical order of its vertices. We now define the flattened tree, constructed from the original tree by exploding all blobs and keeping all vertices;

\begin{definition}[Flattened tree]\label{def:flattenedtree}
	Let $T$ be a multitype tree. The flattened tree $T^{\flat}$ is obtained from the reduced tree $T^{\redu}$ by connecting, for all $x \in T$ of type $1$, the elements of $A_x \setminus (\{x\} \cup A^{(1)}_x)$ to $x$. We give $T^{\flat}$ the only lexicographical order such that, for any $x \in T^{\flat}$ of type $1$, and any $1 \leq k_1 \leq k_2 \leq k_x(T^{\flat})$, $e_{T^{\flat}}(xk_1) \leq e_{T^{\flat}}(xk_2)$ and order the non type 1 vertices according to any fixed but arbitrary ordering inside each type.
	
	In particular, $\#_i(T^{\flat})=\#_i(T)$ for all $i \in [K]$, and the vertices of type $j \neq 1$ in $T^{\flat}$ are leaves.
\end{definition}

\subsection{Flattened Projection}
We can define the flattened projection of a family of probability distributions.

\begin{definition}[Flattened projection]\label{def:flattenedprojection}
	Let $\bzeta$ be a $K$-type offspring distribution, and $\cT \in \bT^{(K,1)}$ be a $\bzeta$-\bgw tree. 
	We define the flattened projection $\btmu$ of $\bzeta$ as the probability distribution on $\N_0^K$ such that, for all $k_1, \ldots, k_K \geq 0$,
	\begin{align*}
		\btmu(k_1, \ldots, k_K) = \Prb{(\#_1(B_\varnothing), \ldots, \#_K(B_\varnothing) )=(k_1+1,k_2,\ldots, k_K)}.
	\end{align*}
	Notice that $\btmu$ is a distribution on $\N_0^K$, corresponding to the distribution of the vertex types of the children of the root in the blob $B_\varnothing$. It is thus also a projection of the distribution $\btzeta$, when $\btzeta$ denotes the offspring distribution of the flattened tree of a $K$-\bgw tree with offspring distribution $\bzeta$. We denote by $\tmu_1$ the measure on $\N_0$ with $\tmu_1(k)=\Prb{\#_1(B_\varnothing)=k}$.
\end{definition}

This leads to a slightly different recursive construction of the flattened tree according to this probability distribution $\btmu$. Consider a family of i.i.d. variables $(X_u, u \in \cU)$ with values in $\N_0^{K}$, with law $\btmu$. We recursively construct a random $K$-type tree $\cT^{(1)} := (t, e_t)$ with root type $1$ as follows. First, $\varnothing \in t$ and $e_t(\varnothing) = 1$. Next, inductively, if $u \in t$ and $e_t(u) = 1$, then write $X_u := (k_1, k_2,\dots,k_K)$. Then for $k \in \N$, $uk \in t$ if and only if $1\leq k \leq \sum_{i=1}^Kk_i$ and $e_t(uk)$ is equal to the smallest index $j$ such that $\sum_{i=1}^jk_i\geq k$. Note that if $e_t(u)\neq 1$ then $u$ is a leaf of the tree (indeed, in the flattened tree only the type-1 vertices are fertile).

We will always work under the natural coupling induced by the flattening and reduction operations.

\label{def:treeandzetaproperties}
We denote by $\bxi := (\xi^{(1)}, \ldots, \xi^{(K)})$  a random vector with distribution $\bmu$ as in Definition~\ref{def:projection}, and $\xi^{(i)}=(\xi^{(i)}_1, \ldots, \xi^{(i)}_K)$ is then distributed according to $\mu^{(i)}$. 

We also denote by $\btxi:=(\txi_1, \ldots, \txi_K)$ \label{def:xis} a random vector on $\N_0^K$ with distribution $\btmu$ (the projection of $\btzeta$). Finally, we denote by $\cT_n$ a $\bzeta$-\bgw tree $\cT$ with root type $1$ conditioned on the event that $\#_{\blambda} \cT=n$ (provided that this holds with positive probability), and by $\ctT_n$ the associated flattened tree. 
Observe that the $K\times K$ mean matrix $M$ defined earlier satisfies: for all $1 \leq i,j \leq K$,
\begin{align*}
	M_{i,j}=\mathbb{E}[\xi^{(i)}_j].
\end{align*}

We illustrate the reduction and flattening operations in Figures~\ref{fig:operationsontrees} and~\ref{fig:flat2}. Observe that the reduced tree of the flattened tree is the reduced tree itself, and is also a subtree of the flattened tree. It is not hard to verify that if $T$ is a $K$-type \bgw tree with offspring distribution $\bzeta$, then the flattened tree is again a $K$-type \bgw tree, and we denote by $\btzeta$ its offspring distribution. 

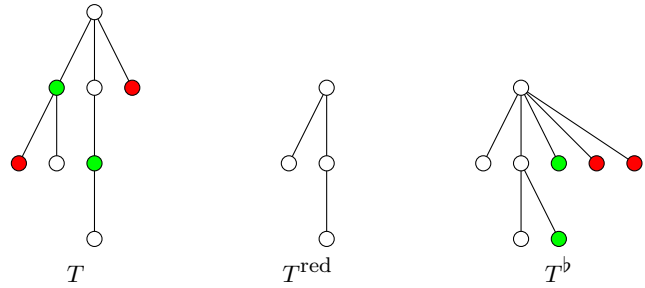
\begin{figure}[!ht]
	\centering
	\begin{tabular}{c c c c c}
		\begin{tikzpicture}
			\draw (-1,-2) -- (-0.5,-1) -- (0,0) -- (0,-3) (.5,-1) -- (0,0) (-.5,-1) -- (-.5,-2);
			\draw[fill=white] (0,0) circle(.1);
			\draw[fill=white] (0,-1) circle(.1);
			\draw[fill=white] (-.5,-2) circle(.1);
			\draw[fill=white] (0,-3) circle(.1);
			\draw[fill=green] (-.5,-1) circle(.1);
			\draw[fill=green] (0,-2) circle(.1);
			\draw[fill=red] (-1,-2) circle(.1);
			\draw[fill=red] (.5,-1) circle(.1);
		\end{tikzpicture}
		&  
		\begin{tikzpicture}
			\draw[white] (0,0) -- (1,0);
		\end{tikzpicture}
		&
		\begin{tikzpicture}
			\draw (-0.5,-1) -- (0,0) -- (0,-2);
			\draw[fill=white] (0,0) circle(.1);
			\draw[fill=white] (0,-1) circle(.1);
			\draw[fill=white] (-.5,-1) circle(.1);
			\draw[fill=white] (0,-2) circle(.1);
		\end{tikzpicture} 
		&  
		\begin{tikzpicture}
			\draw[white] (0,0) -- (1,0);
		\end{tikzpicture}
		&  
		\begin{tikzpicture}
			\draw (-0.5,-1) -- (0,0) -- (0,-2) (.5,-1) -- (0,0) -- (1,-1) (1.5,-1) -- (0,0) (.5,-2) -- (0,-1);
			\draw[fill=white] (0,0) circle(.1);
			\draw[fill=white] (0,-1) circle(.1);
			\draw[fill=white] (-.5,-1) circle(.1);
			\draw[fill=white] (0,-2) circle(.1);
			\draw[fill=green] (.5,-1) circle(.1);
			\draw[fill=green] (.5,-2) circle(.1);
			\draw[fill=red] (1.5,-1) circle(.1);
			\draw[fill=red] (1,-1) circle(.1);
		\end{tikzpicture}
		\\
		$T$   &  
		\begin{tikzpicture}
			\draw[white] (0,0) -- (1,0);
		\end{tikzpicture} & $T^{\redu}$      &  
		\begin{tikzpicture}
			\draw[white] (0,0) -- (1,0);
		\end{tikzpicture} & $T^{\flat}$
	\end{tabular}
	\caption{Left: a tree $T$ with $3$ types. Type $1$ is represented in white, type $2$ in green and type $3$ in red. Middle: the reduced tree $T^{\redu}$. Right: the flattened tree $T^{\flat}$.}
	\label{fig:operationsontrees}
\end{figure}

\begin{figure}[!ht]
	\centering
	\begin{tabular}{c c c c c}
		\begin{tikzpicture}
			\draw (-1,-2) -- (-0.5,-1) -- (0,0) -- (0,-2) (.5,-1) -- (0,0) (-.5,-1) -- (-.5,-2) (0,0)--(-1,-1);
			\draw[fill=white] (0,0) circle(.1);
			\draw[fill=white] (0,-1) circle(.1);
			\draw[fill=white] (-.5,-2) circle(.1);
			\draw[fill=green] (-.5,-1) circle(.1);
			\draw[fill=white] (0,-2) circle(.1);
			\draw[fill=red] (-1,-2) circle(.1);
			\draw[fill=red] (.5,-1) circle(.1);
			\draw[fill=green](-1,-1) circle(.1);
		\end{tikzpicture}
		&  
		
		&
		\begin{tikzpicture}
			\draw (-1,-2) -- (-0.5,-1) -- (0,0) -- (0,-2) (.5,-1) -- (0,0) (-.5,-1) -- (-.5,-2) (0,0)--(-1,-1);
			\draw[fill=white] (0,0) circle(.1);
			\draw[fill=white] (0,-1) circle(.1);
			\draw[fill=white] (-.5,-2) circle(.1);
			\draw[fill=red] (-.5,-1) circle(.1);
			\draw[fill=white] (0,-2) circle(.1);
			\draw[fill=red] (-1,-2) circle(.1);
			\draw[fill=green] (.5,-1) circle(.1);
			\draw[fill=green](-1,-1) circle(.1);
		\end{tikzpicture}
		&  
		\begin{tikzpicture}
			\draw[white] (0,0) -- (1,0);
		\end{tikzpicture}
		&  
		\begin{tikzpicture}
			\draw (-0.5,-1) -- (0,0) -- (0,-2) (.5,-1) -- (0,0) -- (1,-1) (1.5,-1) -- (0,0) (0,0)--(2,-1);
			\draw[fill=white] (0,0) circle(.1);
			\draw[fill=white] (0,-1) circle(.1);
			\draw[fill=white] (-.5,-1) circle(.1);
			\draw[fill=white] (0,-2) circle(.1);
			\draw[fill=green] (.5,-1) circle(.1);
			\draw[fill=green] (1,-1) circle(.1);
			\draw[fill=red] (1.5,-1) circle(.1);
			\draw[fill=red] (2,-1) circle(.1);
		\end{tikzpicture}
		\\
		$T_1$   &  
		\begin{tikzpicture}
			\draw[white] (0,0) -- (1,0);
		\end{tikzpicture} & $T_2$      &  
		\begin{tikzpicture}
			\draw[white] (0,0) -- (1,0);
		\end{tikzpicture} & $T_1^{\flat}=T_2^{\flat}$
	\end{tabular}
	\caption{Left and middle: two different trees $T_1$ and $T_2$ with $3$ types, which lead to the same flattened tree on the right. Type $1$ is represented in white, type $2$ in green and type $3$ in red. Right: the common flattened tree $T_1^{\flat}=T_2^{\flat}$.}
	\label{fig:flat2}
\end{figure}
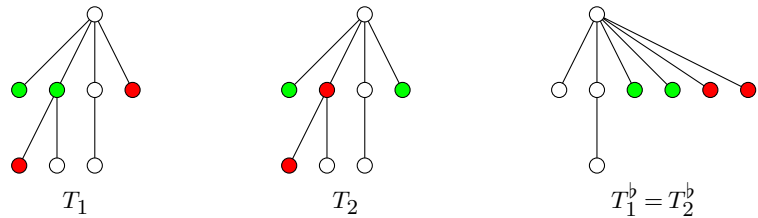

\subsection{Contour function and height function}
\label{def:contour height}
We define here two ways to code a plane tree by a function, which will be used in the rest of the paper: the \textit{contour function} and the \textit{height function}.

\paragraph*{Contour function}

To define the contour function $C_t$ of a plane tree $t$, we consider the function $f_t:\{0,\dots,2(|t|-1)\} \to t$ constructed as follows. First, $f_t(0)=v_1$ (the root). Now for $i<2(|t|-1)$, given $f_t(i)=v$, 
if all children of $v$ have already been visited, then we define $f_t(i+1)$ as the parent of $v$; otherwise, define $f_t(i+1)$ as the smallest child of $v$ (in the lexicographical order) which was not already visited, i.e. which does not belong to $\{f_t(j), j \leq i\}$. The contour function of $t$ is then the function $(C_t(x), x \in [0,2(|t|-1)])$, where $C_t(i)$ is the height of $f_t(i)$ for all $i \in \{0,\dots,2(|t|-1)\}$ and interpolated linearly in between.

\paragraph*{Height function}
Let $t$ be a plane tree, and recall that $v_1(t), \ldots, v_{|t|}(t)$ is the list of the vertices of $t$ listed in lexicographical order. The height function of $t$, denoted by $H_t$, is given as follows. For all $i \in \llbracket 0,|t|-1 \rrbracket$, $H_t(i)=h(v_{i+1}(t))$ is the height of the vertex $v_{i+1}(t)$, and $H_t$ is then extended by linear interpolation to $[0,|t|-1]$.

\subsection*{Frequently used terminology for Sections~\ref{sec:cvreduced},~\ref{sec:cvflattened} and~\ref{sec:cvoriginal}}

\begin{itemize}
	\item $K$: number of types. Type $1$ is the root type. 
	\item $m_T$ mass  measure on a finite tree $T$, see page~\pageref{def:massmeasure}.
	\item $C_T$ contour function of a tree $T$, see page~\pageref{def:contour height}.
    \item $H_T$ height function of a tree $T$, see page~\pageref{def:contour height}.
	\item $\bzeta := \left(\zeta^{(1)}, \ldots, \zeta^{(K)} \right)$: original probability distribution (ordered), see page~\pageref{def:mgwzeta} for the construction of a tree from this and page~\pageref{def:treeandzetaproperties} for further assumptions. This probability distribution is assumed to be critical, irreducible and to have finite exponential moments. 
	\item $\bmu := \left( \mu^{(1)}, \ldots, \mu^{(K)} \right)$: projection of $\bzeta$, see page~\pageref{def:projection}.
	\item $\btmu $ probability measure on $\N_0^{K}$, flattened projection of $\bzeta$, see page~\pageref{def:flattenedprojection}.
    \item $\tmu_1$ probability measure on $\N_0$, see page~\pageref{def:flattenedprojection}. 
	\item $\bxi := \left(\xi^{(1)}, \ldots, \xi^{(K)} \right)$: random vector  with law $\bmu$, see page~\pageref{def:xis}.
	\item $\btxi := \left(\txi_1, \ldots, \txi_K \right)$: random vector with law $\btmu$, see page~\pageref{def:xis}.
	\item $((\txi_i^{(j)})_{i \in [K]})_{j \geq 1}$ independent copies of $(\txi_i)_{i \in [K]}$
	\item $\#_i(T)$: number of vertices of type $i$ in tree $T$.
	\item $N_d(T)$: number of vertices with $d$ children of type $1$ in tree $T$.
	\item $|T|$ size of tree $T$, i.e. the number of vertices of $T$.
	\item $\cT$ a $\bm{\zeta}$-Bienaym\'e tree, see page~\pageref{def:treeandzetaproperties}.
	\item $\cT_n$: $(\cT | \sum \lambda_i \#_i(\cT)=n)$
	\item $\ctT$: the flattened tree, see page~\pageref{def:flattenedtree}.
	\item $\ctT_n$: $(\ctT | \sum \lambda_i \#_i(\ctT)=n)$
	\item $T^{\redu}$: the (root-type) reduced tree, see page~\pageref{def:reducedtree}. 
	\item $c_1:=\sum\limits_{i=1}^K\lambda_i\mathbb{E}[\txi_i]$.
	\item $h_{T}(v)$: height of the vertex $v$ of tree $T$, also denoted $h(v)$ if the tree is clear.
	\item $\Phi$ the blow-up operation, see page~\pageref{def:blowupoperation}.
\end{itemize}

Terminology in Sections~\ref{sec:reducible} and~\ref{sec:types} is analogous, but types range from $1$ to $K+K'$ instead due to additional subcritical components.

\section{Convergence of the reduced tree}
\label{sec:cvreduced}

In what follows, $\bzeta$ is a $K$-tuple of probability distributions with finite exponential moments, which is irreducible and critical. Recall that we denote by $\cT_n$ the $\bzeta$-\bgw tree with root type $1$, conditioned on $\#_{\blambda}(\cT) = n$. We will only consider integers $n$ for which the probability of this event is positive. Recall that $\kT_n$ is the reduced tree of $\cT_n$ defined above, and that $\kT_n$ does not have $n$ vertices, but a random number $\#_1(\cT_n)$.

We define a constant of interest, which appears in the renormalization of the tree $\kT_n$:
\begin{equation}
	\label{eq:kappatree}
	\kappa_{tree}:=\frac{1}{2} \sqrt{\mathbb{V}[\txi_1] \, \sum\limits_{i=1}^K \lambda_i \mathbb{E}[\txi_i]}.
\end{equation}

The goal of this section is to prove convergence of the reduced tree of a size-conditioned multitype tree.

\begin{theorem}
	\label{thm:ConvRed}
	The following convergence holds in distribution, for the rooted Gromov--Hausdorff--Prokhorov distance:
	\begin{align*}
		(\kT_n,\kappa_{tree}n^{-\frac{1}{2}}d_{\kT_n},m_{\kT_n})\convd (\mathcal{T}_{\be},d_{\mathcal{T}_{\be}},m_{\mathcal{T}_{\be}}).
	\end{align*}
\end{theorem}
The key idea of the proof is that, if we condition the reduced tree $\kT_n$ to have exactly $\ell$ vertices for some admissible $\ell \geq 1$, then $(\kT_n | \# \kT_n = \ell)$ is a mixture of trees with $\ell$ vertices and prescribed degree sequences, which we define as follows.

\begin{definition}[Trees with prescribed degree sequences]
Take $r \geq 1$ and let $\mathbf{k} := ((\ell_1,k_1), \ldots, (\ell_r,k_r))$ be a sequence of $r$ elements of $[K] \times \N_0^K$. We say that $\mathbf{k}$ is admissible if, for all $j \in [K]$,
\begin{align*}
	\sum_{i=1}^r k_i(j) = \sum_{i=1}^r \1_{\ell_i=j} - \1_{j=1}.
\end{align*}

A sequence $\mathbf{k}$ is admissible if and only if it is the degree sequence of a multitype tree with root type $1$. Letting $\mathbf{k}$ be an admissible sequence, we define $\cT_{\mathbf{k}}$ as the random tree uniformly distributed among all trees in $\mathbb{T}^{(K,1)}$ with degree sequence $\mathbf{k}$. In particular, $\cT_{\mathbf{k}}$ always has root type $1$.
\end{definition}

\subsection{Distribution of the reduced tree}

We recall the following result, which can be found in~\cite{Mie08}, Proposition 4(i),(iii) . 

\begin{lemma}
	\label{lem:meanexpomoments}
	$\mathbb{E}[\txi_1]=1$, $\bm{\txi}$ has finite exponential moments, and $\Prb{\txi_1=0}>0$.
\end{lemma}

Observe that our reduction operation corresponds to Miermont's projection in \cite{Mie08} and that, if the reduced tree is critical, then the flattened tree is also a critical $K$-type \bgw tree.

The key tool in the proof of Theorem \ref{thm:ConvRed} is the following proposition, which gathers estimates about the distribution of the reduced tree $\kT_n$.
Define the constant
\begin{align*}
	c_1:=\sum\limits_{i=1}^K \lambda_i \mathbb{E}\left[\tilde{\xi}_i\right]=\sum\limits_{i=2}^K \lambda_i \mathbb{E}\left[\tilde{\xi}_i\right]+\lambda_1
\end{align*}
We also denote by $N_d(T)$, for $d \geq 0$ and any tree $T$, the number of vertices in $T$ having exactly $d$ children of type $1$. In particular, $N_d(\kT_n)$ is simply the number of vertices in $\kT_n$ with outdegree $d$, and $N_d(\kT_n)=N_d(\ctT_n)$.

\begin{proposition}
	\label{prop:concentration}
	\begin{enumerate}
		\item[(i)] Fix $0<\delta <1/2$. There exist constants $A,B>0$ independent of $n$ such that, for all $n \geq 1$: 
		\begin{align*}
			\mathbb{P}\left( \left|{\#_1 (\cT_n)}-\frac{n}{c_1}\right| > n^{\frac{1}{2}+\delta} \right) \leq A \exp\left( -B n^\delta \right).
		\end{align*}
		\item[(ii)] For each integer $d\geq 0$, we have 
		\begin{align*}
			N_d(\kT_n)/n \convp \mathbb{P}(\txi_1=d)/c_1
		\end{align*}
		\item[(iii)] There exists a constant $C>0$ such that
		\begin{align}
			\label{eq:eq3}
			\lim\limits_{n\to\infty}\Prb{\sum\limits_{d\geq C\ln(n)}N_d(\kT_n)=0}=1.
		\end{align}
		\item[(iv)] It holds that 
		\begin{align*}
			c_1 \, \sum_{d\geq 1} d^2 \frac{N_d(\kT_n)}{n} \convp \mathbb{E}[\txi_1^2].
		\end{align*}
		\item[(v)] For each $i \in [K]$ we have
		\[
			\#_i(\cT_n) / n \convp \frac{\Ex{\tilde{\xi}_i}}{c_1}.
		\]
	\end{enumerate}
\end{proposition}
\begin{proof}
	Let us start by proving $(i)$. Let $((\txi_i^{(j)})_{i \in [K]})_{j \geq 1}$ be independent copies of $(\txi_i)_{i \in [K]}$. By the irreducibility assumption we have that $c_1>0$, since at least one of the ($\lambda_i, i\in [K]$) is positive.
	Observe that we can construct a tree distributed as the flattened tree $\ctT$ as follows: the root has type $1$ and children distributed as $(\txi_i^{(1)})_{i \in [K]}$, ordered by type. Then, independently, its $\txi_1^{(1)}$ children of type $1$ respectively have children according to $(\txi_i^{(2)})_{i \in [K]}, \ldots, (\txi_i^{(\txi_1^{(1)}+1)})_{i \in [K]}$, and so on.
	Observe that, for any $\ell \geq 1$, the tree has $\ell$ vertices of type $1$ if and only if $\sum\limits_{j=1}^{\ell}(\txi_1^{(j)}-1)=-1$ and $\sum\limits_{j=1}^m(\txi_1^{(j)}-1)\geq 0 \text{ for all } 1\leq m <\ell$. This entails:
	\begin{align}
		\label{eq:eq1}
		&\Prb{\#_1(\cT)=\ell,\#_{\blambda}(\cT)=n}  \\
		&\!=\Prb{\!\left(\sum\limits_{j=1}^{\ell}(\txi_1^{(j)}-1), \sum\limits_{j=1}^{\ell}\sum\limits_{i=2}^K \lambda_i\txi_i^{(j)}\right)\!=(-1,n- \lambda_1 \ell), \sum\limits_{j=1}^m(\txi_1^{(j)}-1)\geq 0 \text{ for all } 1\leq m <\ell\!} \nonumber\\
		&=\frac{1}{\ell}\Prb{\left(\sum\limits_{j=1}^\ell(\txi_1^{(j)}-1),\sum\limits_{j=1}^{\ell} \sum\limits_{i=2}^K \lambda_i\txi_i^{(j)}\right)=(-1,n-\lambda_1 \ell)} \nonumber
	\end{align}
	where the last equality follows from the cycle lemma, see e.g. ~\cite{Janson12}, Lemma 15.3.

	We now claim that there exist constants $A_1,B_1>0$ such that, uniformly in $\ell \geq 1$, $n \geq 1$, we have:
	\begin{align}\label{dev}
		\Prb{\left(\sum\limits_{j=1}^\ell(\txi_1^{(j)}-1),\sum\limits_{j=1}^{\ell} \sum\limits_{i=2}^K \lambda_i\txi_i^{(j)}\right)=(-1,n-\lambda_1 \ell)}\leq A_1 \exp\left(-B_1 \frac{\left|n-c_1\ell\right|}{\sqrt{\ell}}\right).
	\end{align}
	To prove it, observe that, 
by Lemma~\ref{DeviationLemma}, there exists $r,D>0$ such that, for any $\varrho \in (0,r)$, for all $\ell \geq 1$, we have
	\begin{align*}
		\Prb{\left(\sum\limits_{j=1}^\ell(\txi_1^{(j)}-1),\sum\limits_{j=1}^{\ell} \sum\limits_{i=2}^K \lambda_i\txi_i^{(j)}\right)=(-1,n-\lambda_1 \ell)}  \\ &\hspace{-3cm}\leq \Prb{\left|\sum\limits_{j=1}^{\ell}\left(\sum\limits_{i=2}^K\lambda_i\txi_i^{(j)}-c\right)\right|=\left|n-\lambda_1 \ell -c\ell \right|}\\
		&\hspace{-3cm}\leq 2 \exp\left(D \varrho^2 \ell - \varrho  |n-c_1 \ell|\right),
	\end{align*}
	with $c:=\sum\limits_{i=2}^K \lambda_i \mathbb{E}\left[\tilde{\xi}_i\right]=c_1-\lambda_1$. In particular, choosing $\varrho = r \ell^{-1/2}$, we get:
	\begin{align*}
		\Prb{\#_1(\cT)=\ell,\#_{\blambda}(\cT)=n} 
		&\leq \frac{2}{\ell} \exp\left(D \varrho^2 \ell - \varrho |n-c_1 \ell|\right)  \\
		&= \frac{2}{\ell} \exp\left(D r^2 - r\frac{|n-c_1 \ell |}{\sqrt{\ell}}\right).
	\end{align*}
	This satisfies~\eqref{dev} for $A_1=2 \exp(D r^2)$ and $B_1=r$. 
	
	Note also that in case $\ell \ge 2n/c_1$ we may select $\varrho>0$ sufficiently small so that $D \varrho^2  - \varrho c_1 /2<0$, yielding
	\begin{align}
		\label{eq:doit}
				\Prb{\#_1(\cT)=\ell,\#_{\blambda}(\cT)=n} &\leq \frac{2}{\ell} \exp\left(D \varrho^2 \ell - \varrho |n-c_1 \ell|\right) \nonumber  \\
				&\le \frac{2}{\ell} \exp\left( (D \varrho^2  - \varrho c_1 /2)\ell \right) \nonumber \\
				&= \exp(- \Theta(\ell) ).
	\end{align}
	
	Now, fix $0<\delta< 1/2$ and set
	\begin{align*}
		I_\delta=\left\{\ell \, : \, \left|c_1\ell-n\right|<n^{\frac{1}{2}+\delta}\right\}.
	\end{align*}
	Observe that, by e.g.~\cite{Ste18}, Section $4.3.1$, we have the following polynomial asymptotic: 
	\begin{equation}
		\Prb{\#_{\blambda}(\cT)=n}\sim c_{\blambda}n^{-\frac{3}{2}}\label{poldecay}
	\end{equation}
	for some constant $c_{\blambda}>0$ and $\Prb{\#_{\blambda}(\cT)=n}>0$ for infinitely many $n$. Using \eqref{poldecay} along with \eqref{eq:eq1} and \eqref{dev}, we obtain
	\begin{align*}
		\Prb{\#_1(\cT_n)\notin I_\delta}&=\Prb{\#_1(\cT)\notin I_\delta|\#_{\blambda}(\cT)=n}  \\ &=\frac{1}{\Prb{\#_{\blambda}(\cT)=n}}\Prb{\#_1(\cT)\notin I_\delta, \#_{\blambda}(\cT)=n}  \\
		&\leq O(n^{3/2})\sum\limits_{\ell\notin I_\delta}\frac{1}{\ell}\exp\left(-B \frac{|n-c_1\ell|}{\sqrt{\ell}}\right).
	\end{align*}
	Therefore,
	\begin{align*}
		\Prb{\#_1(\cT_n)\notin I_\delta, \#_1(\cT_n) \le 2n/c_1}
        &\leq O(n^{3/2})\sum\limits_{\substack{\ell=1 \\ \ell\notin I_\delta}}^{2n/c_1}\frac{1}{\ell}\exp\left(-B \frac{|n-c_1\ell|}{\sqrt{\ell}}\right)\\
        &\leq O(n^{3/2})\sum\limits_{\substack{\ell=1 \\ \ell\notin I_\delta}}^{2n/c_1}\frac{1}{\ell}\exp\left(-B \frac{n^{1/2+\delta}}{\sqrt{\ell}}\right)\\
        &\leq O(n^{5/2}) \max_{1 \leq \ell \leq 2n/c_1} \exp\left( -B \frac{n^{1/2+\delta}}{\sqrt{\ell}}\right) \\
       	 &\le O(n^{5/2}) \exp\left(-B\sqrt{c_1/2} \frac{n^{1/2+\delta}}{\sqrt{n}}\right) \\
		&\le A_2 \exp\left(-B_2 n^{\delta} \right)
	\end{align*}
    	for some constants $A_2,B_2>0$. By~\eqref{eq:doit} we obtain likewise
	\begin{align}
		\label{eq:prel}
		\Prb{ \#_1(\cT_n) = \ell} \le \exp\left(-\Theta(\ell)\right)
	\end{align}
	uniformly for all $\ell > 2n/c_1$, and hence
	\begin{align}
		\label{eq:type1bound}
\Prb{ \#_1(\cT_n) \ge  2n/c_1} \le A_3 \exp(- B_3 n)
	\end{align}
	for some constants $A_3, B_3>0$. This proves (i).
	In order to prove (ii), for all $d \geq 0$, let $p_d := \Prb{\txi_1=d}$ and define the interval
	\begin{align*}
		J_{\delta,d}=\left[p_d\frac{n}{c_1}-n^{\frac{1}{2}+2\delta}, p_d\frac{n}{c_1}+n^{\frac{1}{2}+2\delta}\right].     
	\end{align*}
    
	Fix $d \geq 0$ and recall that $N_d(\kT_n)=N_d(\ctT_n)$. We want to show that, with high probability, $N_d(\ctT_n) \in J_{\delta,d}$. If $p_d=0$ then $N_d(\ctT_n)=0$ a.s.\ and our claim holds. Hence we may assume in the following that $p_d>0$.        
	We also consider only $n$ such that $\Prb{\#_{\blambda}\cT=n}>0$.
	Using \eqref{poldecay} and Proposition~\ref{prop:concentration} (i), we get that, as $n \to \infty$:
	\begin{align*}
		\Prb{N_d(\ctT_n)\notin J_{\delta,d}} &\leq \Prb{\#_1(\cT_n) \notin I_\delta} + \sum\limits_{\ell\in I_\delta}\Prb{N_d(\ctT_n)\notin J_{\delta,d},\#_1(\cT_n)=\ell}\\
		&= o(1)+\sum\limits_{\ell\in I_\delta}\Prb{N_d(\ctT_n)\notin J_{\delta,d},\#_1(\cT_n)=\ell}\\
		&=o(1)+\sum\limits_{\ell\in I_\delta}\frac{\Prb{N_d(\ctT)\notin J_{\delta,d}, \#_1(\cT)=\ell,\sum\limits_{i=2}^{K}\lambda_i \#_i (\cT)=n-\lambda_1 \ell}}{\Prb{\#_{\blambda} (\cT)=n}}\\
		&\leq o(1)+O(n^{3/2})\sum\limits_{\ell\in I_\delta}\Prb{\sum\limits_{j=1}^{\ell} 1_{\txi_1^{(j)}=d}\notin J_{\delta,d}, \sum\limits_{j=1}^{\ell}\sum\limits_{i=2}^{K}\lambda_i \txi_i^{(j)}=n-\lambda_1 \ell}
	\end{align*}
	Here $\ctT$ denotes the flattened tree of the unconditioned \bgw tree $\cT$. Fix $0<\delta<1/4$. We may assume that $n$ is large enough such that $n^{\delta}-1>1.$
    We can use a binomial Chernoff bound to conclude that for $\ell$ satisfying that $|\ell-\frac{n}{c_1}|<n^{1/2+\delta}$
\begin{align*}
    \Prb{\left| \sum\limits_{j=1}^{\ell}1_{\tilde{\xi}_1^{(j)}=d}-  \ell p_d \right|\geq n^{-1/2+\delta/2}\ell p_d}&\leq 2\exp\left(-\frac{1}{4}n^{-1+\delta}\ell p_d\right)\\
    &\leq 2\exp\left(-\frac{p_d}{4}n^{-1+\delta}\left(\frac{n}{c_1}-n^{1/2+\delta}\right) \right)\\
    &\leq \exp\left(-\Theta(n^{\delta})\right).    
\end{align*}
Now
	\begin{align*}
		\Prb{ \left| \sum\limits_{j=1}^{\ell}1_{\tilde{\xi}_1^{(j)}=d}-  \frac{n}{c_1}p_d \right| \ge n^{1/2+2\delta}}&\leq \Prb{\left| \sum\limits_{j=1}^{\ell}1_{\tilde{\xi}_1^{(j)}=d}-  \ell p_d \right|+\left|\ell p_d-\frac{n}{c_1}p_d\right|\geq n^{1/2+2\delta}}\\
        &\leq \Prb{\left| \sum\limits_{j=1}^{\ell}1_{\tilde{\xi}_1^{(j)}=d}-  \ell p_d \right|\geq n^{1/2+2\delta}- \left|\ell p_d-\frac{n}{c_1}p_d\right|}\\
        &\leq \Prb{\left| \sum\limits_{j=1}^{\ell}1_{\tilde{\xi}_1^{(j)}=d}-  \ell p_d \right|\geq n^{1/2+2\delta}- n^{1/2+\delta}}\\
        &\leq  \Prb{\left| \sum\limits_{j=1}^{\ell}1_{\tilde{\xi}_1^{(j)}=d}-  \ell p_d \right|\geq n^{1/2+\delta}}.
	\end{align*}
Note that for such $\ell$ we have $n^{-1/2+\delta/2}\ell\leq \frac{c_1+1}{c_1}n^{1/2+\delta}$. Therefore also
\begin{align*}
     \Prb{ \left| \sum\limits_{j=1}^{\ell}1_{\tilde{\xi}_1^{(j)}=d}-  \frac{n}{c_1}p_d \right| \ge n^{1/2+2\delta}}&\leq\exp\left(-\Theta(n^{\delta})\right)
\end{align*}

	So we have that, for all $d\geq 0$ and all integers $\ell$ satisfying $|c_1\ell-n|<n^{1/2+\delta}$, 
	\begin{align*}
		\Prb{\sum\limits_{j=1}^{\ell}1_{\txi_1^{(j)}=d} \notin J_{\delta,d}}\leq \exp\left(-\Theta(n^{\delta})\right)\, .
	\end{align*}
	Therefore it follows that, for all $d \geq 0$:
	\begin{align}
		\label{eq:chernoff}
		\Prb{N_d(\ctT_n)\notin J_{\delta,d}}\leq \exp(-\Theta(n^{\delta/2})).
	\end{align}
	This gives us that
	\begin{align*}
		\frac{N_d(\ctT_n)}{n}\convp \frac{p_d}{c_1},
	\end{align*}
	which proves (ii) (recall $N_d(\kT_n)=N_d(\ctT_n)$).
	
	Let us now prove (iii). For $d_0\geq 0$ and for a tree $T$, let $E_{d_0}(T)$ be the event that $N_d(T)>0$ for some $d\geq d_0$. Uniformly for $d_0 \geq 1$, we have:
	\begin{align*}
	& \Prb{E_{d_0}(\ctT_n)} \\ 
        &\leq \Prb{\#_1(\cT_n)\notin I_\delta}+\sum\limits_{\ell\in I_\delta}\Prb{E_{d_0}(\ctT_n),\#_1(\cT_n)=\ell}\\
		&=o(1)+ \frac{1}{\Prb{\#_{\blambda}(\ctT)=n}} \sum\limits_{\ell\in I_\delta} \Prb{E_{d_0}(\ctT), \#_1(\cT)=\ell,\sum\limits_{i=2}^{K}\lambda_i \#_i (\cT)=n-\lambda_1 \ell}\\
		&\leq o(1)+O(n^{3/2})\sum\limits_{\ell\in I_\delta}\Prb{\sum\limits_{j=1}^{\ell}1_{\txi_1^{(j)}\geq d_0}>0, \sum\limits_{j=1}^{\ell}\sum\limits_{i=2}^{K}\lambda_i \txi_i^{(j)}=n-\lambda_1 \ell},
	\end{align*}
	using \eqref{poldecay}. Thus
	\begin{align*}
		\Prb{E_{d_0}(\ctT_n)}\leq o(1)+O(n^{3/2})\sum\limits_{\ell\in I_\delta}\Prb{\sum\limits_{j=1}^{\ell}1_{\txi_1^{(j)}\geq d_0}>0}\\
		\leq o(1)+O(n^{7/2})\Prb{\txi_1\geq d_0}
	\end{align*}
	by a union bound, since $\sup I_\delta = \Theta(n)$. Now, fix $\varepsilon>0$ such that $\mathbb E[e^{\varepsilon \txi_1}]<\infty$ (which exists since $\txi_1$ has finite exponential moments, see Lemma \ref{lem:meanexpomoments}). In particular, for $d$ large enough, we have $\Prb{\txi_1 \geq d} \leq e^{-\varepsilon d}$. We can thus choose $C>0$ such that $C \varepsilon> 7/2$, we get that
	\begin{align*}
		\Prb{N_d(\ctT_n)>0 \text{ for some } d\geq C \ln(n)} = o(1),
	\end{align*}
	which proves (iii).
	
	Let us finally prove (iv). Fix $C>0$ such that \eqref{eq:eq3} holds, and let $E_n$ be the event that 
	\begin{align*}
		\forall d\geq C \ln(n), N_d(\ctT_n)=0.
	\end{align*}
	Then, $\Prb{E_n} \rightarrow 1$ by (iii).
	Furthermore, under $E_n$, we have
	\begin{align*}
		c_1 \sum\limits_{d\geq1}d^2\frac{N_d(\ctT_n)}{n}=c_1\sum\limits_{d=1}^{C \ln (n)}d^2 \frac{N_d(\ctT_n)}{n}.
	\end{align*}
	By \eqref{eq:chernoff}, we obtain that with high probability, for all $1 \leq d \leq C\ln (n), N_d(\ctT_n)\in J_{\delta,d}$. 
	As a consequence, with high probability, under $E_n$, 
	\begin{align*}
		c_1 \sum\limits_{d=1}^{C \ln (n)}d^2 \frac{N_d(\ctT_n)}{n}=\sum\limits_{d=1}^{C \ln (n)}d^2\frac{c_1}{n}\left(\Prb{\txi_1=d}\frac{n}{c_1}+O(n^{\frac{1}{2}+2\delta})\right)\\
		=\sum\limits_{d=1}^{C \ln (n)}d^2\Prb{\txi_1=d}+\sum\limits_{d=1}^{C \ln (n)}d^2O\left(n^{2\delta-\frac{1}{2}}\right)\\
		= o(1)+\mathbb{E}[\txi_1^2]+O\left(n^{2\delta-\frac{1}{2}} \ln^3 (n)\right)
	\end{align*}
	where, in the first two lines, the $O$ is uniform in $d \leq C \ln(n)$. Taking $\delta>0$ small enough, we thus get
	\begin{align*}
		c_1 \sum_{d\geq 1} d^2 \frac{N_d(\ctT_n)}{n} =\mathbb{E}[\txi_1^2]+o(1).
	\end{align*}
	This ends the proof of (iv).
	
	Finally, we verify (v). The case $i=1$ is exactly (i), since
	$\mathbb{E}[\txi_1]=1$ by Lemma~\ref{lem:meanexpomoments}. Fix now
	$i\in\{2,\ldots,K\}$. Since flattening
	preserves the number of vertices of each type, it is enough to prove the
	claim for $\ctT_n$.
	Let $v_1, \ldots, v_{\#_1(\ctT_n)}$ be the vertices of type $1$ of $\ctT_n$ in depth-first-search order.  For each $1 \leq j \leq \#_1(\ctT_n)$, let $Y_j$ denote the number of type $i$ children of $v_j$ in $\ctT_n$.
	
	Fix $0<\delta<1/4$. For each $\epsilon>0$  let $\mathcal{E}$ denote the  event that \[
	\left|\sum_{j=1}^{\#_1(\ctT_n)} Y_j - n \Ex{\tilde{\xi}_i}/c_1  \right| \ge \epsilon n.
	\]
	Let $\mathcal{E}'$ denote the corresponding event for the unconditioned tree $\ctT$. Using \eqref{poldecay} we obtain
	\begin{align*}
		\Prb{\mathcal{E},\#_1(\ctT_n)\in I_\delta} &=\frac{\Prb{\mathcal{E}',\#_1(\ctT)\in I_\delta, \#_{\blambda}\ctT=n}}{\Prb{\#_{\blambda}(\ctT)=n}}\\
		&= O(n^{\frac{3}{2}}) \Prb{\mathcal{E}',\#_1(\ctT)\in I_\delta, \#_{\blambda}(\ctT)=n}\\
		&\leq O(n^{\frac{3}{2}})\sum_{\ell\in I_\delta} \Prb{ \left| \sum\limits_{j=1}^{\ell} \txi^{(j)}_i -n \mathbb{E}[\txi_i]/c_1\right| >\epsilon n} \\
		&= o(1) \text{ by Lemma } \ref{DeviationLemma}.
	\end{align*}
	By $(i)$ it follows that the probability for $\mathcal{E}$ tends to zero as $n \to \infty$. As $\epsilon>0$ was arbitrary, this completes the proof.
 \end{proof}

\subsection{Proof of Theorem~\ref{thm:ConvRed}}

We are now ready to prove Theorem~\ref{thm:ConvRed}. To this end, we start with the following proposition, which will also be useful later on. Given a plane tree $t$, recall the definitions of $C_{t}$ and $H_{t}$ from Definition \ref{def:contour height}. We define their rescaled versions $\hat{C}_{t}$ and $\hat{H}_{t}: [0,1] \rightarrow \R$ as: for all $x \in [0,1]$,
\begin{align*}
\hat{C}_{t}(x)=C_t((2|t|-2)x), \hat{H}_{t}(x)=H_t((|t|-1)x).
\end{align*}

Recall also that we denote by $\be$ a standard Brownian excursion.

\begin{proposition}
\label{prop:jointconvergence}
In the space $\mathcal{C}([0,1],\mathbb{R}^2)$ equipped with the $L_\infty$ distance, we have the following joint convergence in distribution:
\begin{align*}
		\left(\frac{\sqrt{c_1}\hat{C}_{\kT_n}(x)}{\sqrt{n}},\frac{\sqrt{c_1}\hat{H}_{\kT_n}(x)}{\sqrt{n}}\right)_{x \in [0,1]} \convd \left(\frac{2}{\sqrt{\mathbb{V}[\txi_1]}}\be_x,\frac{2}{\sqrt{\mathbb{V}[\txi_1]}}\be_x\right)_{x \in [0,1]}.
	\end{align*}
\end{proposition}

Theorem \ref{thm:ConvRed} actually follows from the convergence of the contour function alone. The joint convergence result will be useful in Section \ref{sec:cvflattened}.

\begin{proof}[Proof of Theorem~\ref{thm:ConvRed}]
The rooted Gromov--Hausdorff--Prokhorov convergence of $\kT_n$ follows from the convergence of the rescaled contour function $\hat{C}_{\kT_n}$ in Proposition \ref{prop:jointconvergence}, see e.g. \cite{ADH14}, Proposition 2.10.
\end{proof}

Let us now prove Proposition \ref{prop:jointconvergence}.

\begin{proof}[Proof of Proposition~\ref{prop:jointconvergence}]
For all $n \geq 1$, define
\begin{align*}
	&\tilde{\nu}_n: d \geq 0 \mapsto \frac{N_d(\kT_n)}{|\kT_n|}\\
	&\tilde{\sigma}_n^2:=\sum_{d\geq 0}d^2\frac{N_d(\kT_n)}{|\kT_n|-1}-1\\
	&\Delta_n:=\max\{d:N_d(\kT_n)>0\}.
\end{align*}
We adapt the ideas of~\cite{BM14}, Theorem 3. First observe that, for any  given integer $\ell$ with $\Prb{ \#_1 (\cT_n) = \ell}>0$, the tree $(\kT_n | \#_1 (\cT_n) = \ell)$ is distributed like a mixture of trees with prescribed degree sequences. Indeed, if we condition the flattened tree $(\ctT_n | \#_1 (\cT_n) = \ell)$ on having a given ($K$-dimensional) profile $\left( N_{a_1, \ldots, a_K}\right)_{a_1, \ldots, a_K \geq 0}$ (that is, the tree having exactly $N_{a_1, \ldots, a_K}$ vertices with $a_i$ offspring of type $i$ for all $1 \leq i \leq K$), then the tree $(\kT_n | \#_1 (\cT_n)~=~\ell)$ is uniform among all trees with the associated profile $(\sum_{a_2, \ldots, a_K \geq 0}N_{a_1, \ldots, a_K})_{a_1 \geq 0}$ (see e.g. \cite{CS23}, Proof of Proposition $11$).

Next, observe that in Proposition ~\ref{prop:concentration} (ii) and (iv) we use $\frac{n}{c_1}$ instead of $|\kT_n|~=~\#_1\ctT_n$ for the scaling, but since (i) yields concentration, we can instead scale by $1/|\kT_n|$ and obtain the same limit. Hence, we get the following limit in distribution as $n\to \infty$:
\begin{equation}
	\label{eq:productconvergence}
	(\tilde{\nu}_n, \tilde{\sigma}_n^2, \frac{\Delta_n}{\sqrt{n}})\convd (\tmu_1,\mathbb{V}[\txi_1],0),
\end{equation}
on the (Polish) space $\mathcal{M}(\mathbb{N}_0)\times \mathbb{R}\times \mathbb{R}$ endowed with the product topology, where $\mathcal{M}(\mathbb{N}_0)$ is the space of probability measures on $\mathbb{N}_0$.

Now by Skorohod's representation theorem we may work in a space in which $ (\tilde{\nu}_n, \tilde{\sigma}_n^2, \frac{\Delta_n}{\sqrt{n}})$ converges almost surely to $(\tmu_1,\mathbb{V}[\txi_1],0)$. Therefore, almost surely the conditions of ~\cite{BM14}, Theorem 3 are satisfied, yielding
\begin{align}
	\label{eq:convergencecontour1}
	\left( \frac{\hat{C}_{\kT_n}(x)}{\sqrt{|\kT_n|}}, \frac{\hat{H}_{\kT_n}(x)}{\sqrt{|\kT_n|}}\right)_{x\in[0,1]} \convd \left(\frac{2}{\sqrt{\mathbb{V}[\txi_1]}}\be_x,\frac{2}{\sqrt{\mathbb{V}[\txi_1]}}\be_x\right)_{x\in[0,1]}.
\end{align}  
Here the convergence holds in the space $\mathcal{C}([0,1],\mathbb{R}^2)$ equipped with the $L_\infty$ distance.

	 Now fix $\delta \in (0,1/2)$. Using Proposition \ref{prop:concentration} (i), with high probability as $n \rightarrow \infty$, we have:
     \begin{align*}
		\left|\frac{\sqrt{c_1}}{\sqrt{n}}-\frac{1}{\sqrt{|\kT_n|}}\right| &\leq \left|\frac{c_1 \cdot n^{1/2+\delta}}{\sqrt{|\kT_n|}\cdot\sqrt{n}(\sqrt{c_1|\kT_n|}+\sqrt{n})}\right| \nonumber \\
		&\leq \frac{c_1 \cdot n^{-1/2+\delta} }{\sqrt{|\kT_n|}}
	\end{align*}

This immediately yields
    
	\begin{align*}
		\sup_{x \in [0,1]} \left|\frac{\sqrt{c_1}\hat{C}_{\kT_n}(x)}{\sqrt{n}}-\frac{\hat{C}_{\kT_n}(x)}{\sqrt{|\kT_n|}}\right| \leq c_1 \cdot n^{-1/2+\delta} \sup_{x \in [0,1]}  \frac{\hat{C}_{\kT_n}(x)}{\sqrt{|\kT_n|}}.
	\end{align*}
	
	Using \eqref{eq:convergencecontour1}, it follows that the left-hand side goes to $0$ in probability, since $\delta<1/2$. The same holds \textit{verbatim} for $\hat{H}_{\kT_n}$. Hence, the joint convergence of Proposition \ref{prop:jointconvergence} follows from \eqref{eq:convergencecontour1}.
\end{proof}

We may also derive a tail bound for the height, which is defined by $\mathrm{H}(T):=\max_{v\in T}h_T(v)$ for any tree $T$. 

\begin{proposition}
	\label{pro:goup}
	There exist constants $C,c>0$ such that for all $n$ and $x>0$
	\[
	\Prb{\mathrm{H}(\ctT_n)\geq x}\leq C\exp(-cx^2/n) + C\exp(-cx).
	\]
\end{proposition}	
\begin{proof}
	The distance of any vertex of $\ctT_n$ to the subtree $\kT_n$ is at most $1$. By possibly adjusting $C,c>0$ it hence suffices to show the tail bounds for the height of $\kT_n$. Furthermore, again by possibly adjusting $C,c>0$, it suffices to verify these bounds for
	\begin{align}
		\label{eq:xas}
		x \ge \sqrt{n}.
	\end{align}

	Set $p_1 := \Prb{\tilde{\xi}_1 = 1}$. We may assume $p_1>0$, since otherwise $\Prb{N_1(\kT_n)>0}=0$.
	For any  given integer $\ell$ with $\Prb{ \#_1 (\cT_n) = \ell}>0$ and any $\epsilon>0$ we may argue analogously as for~\eqref{eq:chernoff} using the Chernoff bounds to obtain
	\begin{align*}
		&\Prb{N_1(\kT_n) \ge (1+ \epsilon) p_1 \ell, | \kT_n|= \ell, \mathrm{H}(\kT_n)\geq x} \\
		&\le \1_{x \le \ell} \Prb{N_1(\kT_n) \ge (1+ \epsilon) p_1 \ell, | \kT_n|= \ell} \\
		&= \1_{x \le \ell}	\Prb{\#_{\blambda} (\cT)=n}^{-1}	\Prb{N_1(\kT) \ge (1+ \epsilon)p_1\ell, | \kT|= \ell,\#_{\blambda} (\cT)=n } \\
		&\le \1_{x \le \ell}	O(n^{3/2})	\Prb{N_1(\kT) \ge (1+ \epsilon) p_1 \ell, | \kT|= \ell} \\
		&\le \1_{x \le \ell}	O(n^{3/2})	\Prb{\mathrm{Bin}(\ell,p_1) \ge (1+ \epsilon) p_1 \ell} \\
		&\le \1_{x \le \ell}	O(n^{3/2}) \exp(- \epsilon^2 \ell p_1 / (2+\epsilon)) \\
		&\le O(n^{3/2}) \exp(- \epsilon^2 x p_1 / (2+\epsilon)).
	\end{align*}
	By~\eqref{eq:xas}, this simplifies to
	\begin{align}
		\label{eq:onebound}
		\Prb{N_1(\kT_n) \ge (1+ \epsilon) p_1 \ell, | \kT_n|= \ell, \mathrm{H}(\kT_n)\geq x} &\le O(1) \exp\left(- \frac{\epsilon^2 x p_1(1+o(1))}{ 2+\epsilon}\right)  \\
		&\le \exp(- \Theta(x) ). \nonumber
	\end{align}

	As we argued in the proof of Theorem~\ref{thm:ConvRed}, for any  given integer $\ell$ with $\Prb{ \#_1 (\cT_n) = \ell}>0$, the tree $(\kT_n \,\,|\,\, | \kT_n| = \ell)$ is distributed like a mixture of random trees with prescribed degree sequences. Likewise, $(\kT_n \,\,|\,\, |\kT_n| = \ell, N_1(\kT_n) < (1+ \epsilon) p_1 \ell)$ is distributed like a mixture of random trees with prescribed degree sequences with $\ell$ vertices such that the number of vertices with outdegree $1$ is  less than $(1+ \epsilon) p_1 \ell$.
	
	Hence, taking $\epsilon>0$ small enough so that $(1+\epsilon)p_1 < 1$, we may apply~\cite{MR4815972}, Theorem 6, yielding that there exist constants $C',c'>0$  that depend on $(1+\epsilon)p_1$, but not on $\ell$, $n$ or $x$ such that 
	\begin{align}
		\label{eq:AA}
		\Prb{\mathrm{H}(\kT_n) \geq x \mid \left| \kT_n\right|=\ell, N_1(\kT_n) < (1+ \epsilon) p_1 \ell}\leq C'\exp(-c'x^2/\ell).
		\end{align}
	Therefore,
	\begin{align*}
		\Prb{\mathrm{H}(\kT_n) \geq x  }  
		= A + B,
	\end{align*}
	with
	\begin{align*}
		A &:= \sum_{1 \le \ell \le 2n/c_1} \Prb{ \mathrm{H}(\kT_n) \geq x, |\kT_n|=\ell }  \\
		&\le (2n/c_1) \exp(- \Theta(x) ) + \sum_{1 \le \ell \le 2n/c_1} \Prb{\mathrm{H}(\kT_n) \geq x ,|\kT_n|=\ell, N_1(\kT_n) < (1+ \epsilon) p_1 \ell } \\
		&\le \exp(- \Theta(x) )  + \sum_{1 \le \ell \le 2n/c_1} \Prb{|\kT_n|=\ell, N_1(\kT_n) < (1+ \epsilon) p_1 \ell} C' \exp\left(- c' x^2 / \ell \right) \\
		&\le \exp(- \Theta(x) ) + C' \exp(-(c' c_1 / 2) x^2 / n)
	\end{align*}
	by Inequalities~\eqref{eq:onebound},~\eqref{eq:xas} and ~\eqref{eq:AA} (applied in that order),
	and
	\begin{align*}
		B &:= \sum_{\ell > 2n/c_1} \1_{x\le \ell} \Prb{|\kT_n|=\ell} \Prb{\mathrm{H}(\kT_n) \geq x \,\mid\, |\kT_n|=\ell }  \\
		&\le \sum_{\ell > 2n/c_1} \1_{x\le \ell} \Prb{|\kT_n|=\ell} \\
		&\le \sum_{\ell > 2n/c_1} \1_{x\le \ell} \exp(-\Theta(\ell)) \\
		&\le \exp(-\Theta(x))
	\end{align*}
	by~\eqref{eq:prel}. The indicator variable $\1_{x\le \ell}$ appears as a factor since by construction $\mathrm{H}(\kT_n) \le |\kT_n|$.
	This completes the proof.
\end{proof}

\section{Convergence of the flattened tree}
\label{sec:cvflattened}

The goal of this section is to prove the convergence of the flattened tree $\ctT_n$, from the convergence of the reduced tree $\kT_n$ shown in the previous section.

\begin{proposition}\label{prop:ConvFlattened}
	Let $\kappa_{tree}$ be the constant defined in \eqref{eq:kappatree}. Then, the following convergence holds in distribution, for the rooted Gromov--Hausdorff--Prokhorov topology.
	\begin{align*}
		(\ctT_n,\kappa_{tree}n^{-\frac{1}{2}}d_{\ctT_n},m_{\ctT_n})\convd (\mathcal{T}_{\be},d_{\mathcal{T}_{\be}},m_{\mathcal{T}_{\be}})
	\end{align*}
\end{proposition}

We observe immediately that the metric part of this convergence is essentially identical to the one of the reduced tree; therefore, most of the work consists in proving the convergence of the mass measures.

\subsection{Maximum outdegree in $\ctT_n$}

We start by showing bounds on the largest outdegree of a vertex of type $1$ in $\ctT_n$. These bounds will help us control how measures change between the reduced tree and the flattened tree.

\begin{lemma}
	\label{lem:largestoutdegree}
	Let $\Delta_n$ be the largest outdegree in the tree $\ctT_n$. Then $\exists \, C>0$
	\begin{equation*}
		\lim_{n \to\infty} \Prb{ \Delta_n \geq C \ln (n) } = 0.
	\end{equation*}
\end{lemma}

\begin{proof}[Proof of Lemma \ref{lem:largestoutdegree}]
	For any vertex $v\in \ctT_n$, let $k_v$ denote the number of children of $v$ in $\ctT_n$. In particular, $k_v=0$ if $v$ is not of type $1$. Let us first assume $\lambda_1>0$. Splitting according to the number of vertices of type $1$ in the tree, we obtain:
	\begin{align*}
		& \Prb{\Delta_n \geq C\ln (n)}\\
		&=\Prb{\exists v\in \ctT_n: k_v \geq C\ln(n)}\\
		&=\frac{\Prb{\exists v\in \ctT: k_v \geq C\ln(n),\#_{\blambda}(\ctT) =n}}{\Prb{\#_{\blambda}(\ctT) =n}}\\
		&= \sum\limits_{1\leq \ell \leq \frac{n}{\lambda_1}} \frac{1}{\Prb{\#_{\blambda}(\ctT) =n}}\Prb{\exists v\in \ctT: k_v \geq C\ln(n),\#_{\blambda}(\ctT) =n, \#_1(\ctT)=\ell}.
	\end{align*}
	
	Since $(\txi_i^{(j)})_{1 \leq i \leq K, 1 \leq j \leq \ell}$ are i.i.d. with finite exponential moments, there exist absolute constants $a,b>0$ such that, for each $j=1,\dots, \ell$, for all $C>0$:
	\begin{align*}
		\Prb{\sum\limits_{i=1}^K\txi_i^{(j)} \geq C \ln(n)} &\leq K \max_{i}\Prb{\txi_i^{(j)} \geq \frac{C}{K}\ln(n)}\\
		&\leq a e^{-b\frac{C}{K}\ln(n)}= a n^{-b\frac{C}{K}}.
	\end{align*}
	We therefore have:
	\begin{align*}
		\Prb{\Delta_n> C\ln (n)}&\leq \sum\limits_{1\leq \ell \leq \frac{n}{\lambda_1}} O(n^{\frac{3}{2}})(1-(1- a n^{-b\frac{C}{K}})^n)\\
		&\leq \frac{n}{\lambda_1}O(n^\frac{3}{2})(1-(1- a n^{-b\frac{C}{K}})^n)\\
		&=O(n^\frac{5}{2}) a n^{1-b\frac{C}{K}} \\
		&\leq O(n^{\frac{7}{2}-b\frac{C}{K}}).
	\end{align*}
    Note that, in the first line, the $O(n^{\frac{3}{2}})$ is uniform over the values of $\ell$ in the sum, and that we have used the fact that $1-(1-x)^n \leq nx$ for all $x \in [0,1]$.
	
	In particular, fixing $C$ large enough, this goes to $0$ as $n \rightarrow \infty$.
	
	On the other hand, in the case where $\lambda_1=0$ we have that the amount of root-type vertices is in $[0,\frac{n}{c_1}+n^{\frac{1}{2}+\frac{1}{10}}]$ with high probability. More precisely, by Proposition~\ref{prop:concentration} (i), setting 
	\begin{align*}
		I=\left\{\ell \, : \, \left|c_1\ell-n\right|<n^{\frac{1}{2}+\frac{1}{10}}\right\},
	\end{align*}
	there exists $A,B>0$ such that
	\begin{align*}
		\Prb{\#_1(\ctT_n)\notin I}\leq A\exp\left(-Bn^{\frac{1}{10}}\right)
	\end{align*}
	This yields
	\begin{align*}
		\Prb{\Delta_n \geq C\ln (n)}
		&=\Prb{\exists v\in \ctT_n, k_v \geq C\ln(n)}\\
		&=\frac{\Prb{\exists v\in \ctT, k_v \geq C\ln(n),\#_{\blambda}(\ctT) =n}}{\Prb{\#_{\blambda}(\ctT) =n}}\\
		&\leq \frac{1}{\Prb{\#_{\blambda}(\ctT) =n}} \\
        &\quad \cdot \sum\limits_{1\leq \ell \leq  \frac{n}{c_1}+n^{\frac{1}{2}+\frac{1}{10}}}\Prb{\exists v\in \ctT, k_v \geq C\ln(n),\#_{\blambda}(\ctT) =n, \#_1(\ctT)=\ell}\\
		&\quad +\Prb{\#_1(\ctT_n)\notin I}\\
		&\leq O(n^{\frac{3}{2}}) \sum\limits_{1\leq \ell \leq \frac{n}{c_1}+n^{\frac{1}{2}+\frac{1}{10}}} (1-(1- a n^{-b\frac{C}{K}})^n)+A\exp\left(-Bn^{\frac{1}{10}}\right)\\
		&\leq (\frac{n}{c_1}+n^{\frac{1}{2}+\frac{1}{10}})O(n^\frac{3}{2})(1-(1- a n^{-b\frac{C}{K}})^n)+ A\exp\left(-Bn^{\frac{1}{10}}\right)\\
		&=O(n^\frac{5}{2}) a n^{1-b\frac{C}{K}}+A\exp\left(-Bn^{\frac{1}{10}}\right) \\
		&\leq O(n^{\frac{7}{2}-b\frac{C}{K}})+A\exp\left(-Bn^{\frac{1}{10}}\right).
	\end{align*}
	Again, for $C$ large enough, this converges to $0$ as $n\to\infty$.
\end{proof}

As an immediate corollary, we get the following:

\begin{corollary}
	\label{cor:sizeoflargestblob}
	Let $\Delta'_n$ be the size of the largest blob in the tree $\cT_n$. Then, there exists $C>0$ such that
	\begin{align*}
		\lim_{n \rightarrow \infty} \Prb{\Delta'_n \geq C \ln(n)} = 0.
	\end{align*}
\end{corollary}
\begin{proof}
	Under the canonical coupling, we have $\Delta'_n \leq\Delta_n+1$. Indeed, the size of the blob associated to a type $1$ vertex $x \in \cT_n$ is $1+$ the number of nonroot-type children of the vertex associated to $x$ in $\ctT_n$.
\end{proof}

\subsection{Proof of Proposition \ref{prop:ConvFlattened}}

In order to deduce Proposition \ref{prop:ConvFlattened} from Theorem \ref{thm:ConvRed}, we need to compare the distances and mass measures in both trees $\kT_n$ and $\ctT_n$. Under the natural coupling, we view here $\kT_n$ as a subtree of $\ctT_n$, so that $\kT_n$ is naturally embedded in $\ctT_n$.

\begin{definition}\label{def:coupledvertices}
	For $u \in ]0,1]$, we define $L_n(u):=\lceil u\#_1(\ctT_n)\rceil$ and $v_n(u)$ the $L_n(u)$'th vertex of $\kT_n$ in the right-reversed depth-first-search order. This is the order we get if instead of choosing the  leftmost child in the depth-first search we choose the rightmost child in every step.
    We also define $w_n(u)$ as the $\lceil u |\ctT_n| \rceil$'th vertex in $\ctT_n$ in  right-reversed depth-first-search order, and $v'_n(u)$ the vertex of type $1$ defined as follows:
	\begin{itemize}
		\item $v'_n(u)=w_n(u)$ if $w_n(u)$ is of type $1$;
		\item $v'_n(u)$ is the parent of $w_n(u)$ otherwise.
	\end{itemize}
	Observe that, by definition of $\ctT_n$, $v'_n(u)$ is always of type $1$.
	Finally, we denote by $L'_n(u) \in [1,\#_1(\ctT_n)]$ the unique integer such that $v'_n(u)$ is the $L'_n(u)$'th type $1$ vertex of $\ctT_n$ in right-reversed depth-first-search order. 
\end{definition}
The next lemma controls the difference between $L_n$ and $L_n'$:
\begin{lemma}\label{lem:coupling}
	For $0< \delta < \frac{1}{4}$, there exists a constant $C>0$ such that, for all $n$:
	\begin{align*}
		\Prb{\sup_{u \in ]0,1]} |L_n(u)-L_n'(u)| > C n^{\frac{3}{4}}} \leq O(n^{-10}).
	\end{align*}
\end{lemma}
\begin{proof}
	Fix $\delta \in (0,\frac{1}{4})$ and recall from Proposition~\ref{prop:concentration} that
	\begin{align*}
		\Prb{\#_1(\ctT_n)\notin I_\delta}\leq \exp\left(-\Theta(n^{\delta})\right)
	\end{align*}
	where $ I_\delta=\left\{\ell \, : \, \left|c_1\ell-n\right|<n^{\frac{1}{2}+\delta}\right\}$. 
	Let $v_1, \ldots, v_{\#_1(\ctT_n)}$ be the vertices of type $1$ of $\ctT_n$ in right-reversed depth-first-search order and, for all $1 \leq i \leq \#_1(\ctT_n)$, denote by $k'_i$ the number of children of $v_i$ in $\ctT_n$ that are not of type $1$.
	
	Now let $\mathcal{E}_\delta$ be the following event: there exists $1\leq j\leq \#_1(\ctT_n)$ such that $\sum_{i=1}^j k'_i \notin J_j$, where 
	$$J_j:=\left\{r \, : \, \left|r-j\sum\limits_{k=2}^K\mathbb{E}[\txi_k]\right|\leq n^{\frac{1}{2}+\delta}\right\}.$$
	We consider again only values of $n$ such that $\Prb{\#_{\blambda}(\ctT)=n}>0$, otherwise there is nothing to show.
	We use again \eqref{poldecay} to get:
	\begin{align*}
		\Prb{\mathcal{E}_\delta,\#_1(\ctT_n)\in I_\delta} &=\frac{\Prb{\mathcal{E}_\delta,\#_1(\ctT)\in I_\delta, \#_{\blambda}\ctT=n}}{\Prb{\#_{\blambda}(\ctT)=n}}\\
		&= O(n^{\frac{3}{2}}) \Prb{\mathcal{E}_\delta,\#_1(\ctT)\in I_\delta, \#_{\blambda}(\ctT)=n}\\
		&\leq O(n^{\frac{3}{2}})\sum_{\ell\in I_\delta}\sum\limits_{j=1}^{\ell}\Prb{\left|\sum\limits_{i=1}^j\sum\limits_{k=2}^K(\txi^{(i)}_k-\mathbb{E}[\txi^{(i)}_k])\right|>n^{\frac{1}{2}+\delta}}\\
        &\leq O(n^{\frac{3}{2}}) \sum_{\ell\in I_\delta}\sum\limits_{j=1}^{\ell} \exp\left( Djn^{-1}-n^{-1/2}n^{1/2+\delta} \right) \text{ by Lemma } \ref{DeviationLemma}\\
	\end{align*}
	for some $D \geq 0$, taking $\varrho=n^{-1/2}$ in Lemma \ref{DeviationLemma}. Hence,
    \begin{align*}
    \Prb{\mathcal{E}_\delta,\#_1(\ctT_n)\in I_\delta} \leq \exp\left(-\Theta(n^\delta)\right).
    \end{align*}
	By Proposition~\ref{prop:concentration}, with high probability $\#_1(\ctT_n)\in I_\delta$. This entails that, with high probability, $\mathcal{E}_\delta$ does not happen. 
    
	In the following we work under the conditioning that $\#_1(\ctT_n)\in I_\delta$ and that $\mathcal{E}_{\delta}$ does not happen and that the maximum outdegree of a root-type vertex is bounded by a fixed constant multiple of $\ln(n)$ (see Lemma \ref{lem:largestoutdegree}), which holds with high probability.
    Note that  
    \begin{align}
    \label{eq:propcouplingevet}
      &\Prb{\sup_{u \in ]0,1]} |L_n(u)-L_n'(u)| > C n^{\frac{3}{4}}} \nonumber\\
      &\leq \Prb{\sup_{u \in ]0,1]} |L_n(u)-L_n'(u)| > C n^{\frac{3}{4}},\#_1(\ctT_n)\in I_\delta, \Delta_n < D \ln (n),\, \mathcal{E}_{\delta}\text { does not happen} } \nonumber \\
      &+\Prb{\#_1(\ctT_n)\notin I_\delta}+\Prb{\Delta_n \geq D\ln (n) }+\exp\left(-\Theta(n^{\delta})\right).
    \end{align}

	To compare $L_n'(u)$ and $L_n(u)$ we retrieve an approximation to $u$ from those integers.
	On the one hand, by definition of $L_n(u)$, dividing $L_n(u)$ by $\#_1(\ctT_n)$ approximates $u$ with approximation error of order $O(n^{-1})$. Recovering $u$ from $L_n'(u)$ is a bit more tricky. To do so, we need to sum over all nonroot type  offspring of the first $L_n'(u)$ vertices in the tree, and add $L_n'(u)$. We then divide the quantity obtained by the total amount of vertices in the tree $\ctT_n$. When doing so, we potentially also add to that quantity vertices which are after $w_n(u)$ in the ordering; these are siblings of $w_n(u)$, if $w_n(u)$ is not of root type, or the children of $w_n(u)$ if $w_n(u)$ is of root type. We may restrict on the event that this amount is of maximum order $O(\ln(n))$ by Lemma~\ref{lem:largestoutdegree}, see also~(\ref{eq:propcouplingevet}).%

   More precisely, setting $S(u):=\sum\limits_{i=1}^{L_n'(u)}\sum\limits_{k=2}^K\txi^{(i)}_k$, we have uniformly in $u \in ]0,1]$:
	
	\begin{align}\label{eq:LU}
		\frac{L_n(u)}{\#_1(\ctT_n)}=u+O(n^{-1})=\frac{S(u)+L_n'(u)}{|\ctT_n|}+O\left(\frac{\ln(n)}{|\ctT_n|}\right).
	\end{align}
    This implies that
    \begin{align}
		\frac{S(u)+L_n'(u)}{|\ctT_n|}=u+O\left(\frac{\ln(n)}{|\ctT_n|}\right).
	\end{align}

	On the other hand, for $u=1$, we have $L_n'(1)=\#_1(\ctT_n)$ and $S(1)+L_n'(1)=S(1)+\#_1(\ctT_n)=|\ctT_n|$. In addition, observe that $|\#_1(\ctT_n)-\frac{n}{c_1}|\leq n^{\frac{1}{2}+\delta}$ (since $\#_1(\ctT_n)\in I_\delta$) and that, for all $u$, $|S(u)-c_2 L_n'(u)|\leq n^{\frac{1}{2}+\delta}$ (since $\mathcal{E}_{\delta}$ does not happen). Here, $c_1$ is (as in Proposition \ref{prop:concentration}) defined as $c_1:=\sum\limits_{i=1}^K \lambda_i \mathbb{E}\left[\tilde{\xi}_i\right]$, and $c_2:=\sum\limits_{i=2}^K\mathbb{E}[\txi_i]$. 
	
	Plugging this into \eqref{eq:LU} entails that $|\ctT_n|=\frac{c_2}{c_1}n+\frac{1}{c_1}n+O( n^{\frac{1}{2}+\delta})=\frac{c_2+1}{c_1}n+O( n^{\frac{1}{2}+\delta})$.
	This yields
	\begin{align}
		\label{eq:approximationofsizeovertype1size}
		\frac{1}{c_2+1}  \frac{|\ctT_n|}{\#_1(\ctT_n)}&=\frac{1}{c_2+1}\frac{\frac{c_2+1}{c_1} n+O(n^{\frac{1}{2}+\delta})}{\frac{n}{c_1}+O(n^{\frac{1}{2}+\delta})} \nonumber \\
		&=\frac{1+O(n^{-\frac{1}{2}+\delta})}{1+O(n^{-\frac{1}{2}+\delta})} \nonumber\\
		&=1+O(n^{-\frac{1}{2}+\delta}).
	\end{align}
	On the other hand, note that, uniformly in $u$:
	\begin{align*}
		S(u)+L_n'(u)=L_n'(u)+c_2 L_n'(u)+O(n^{\frac{1}{2}+\delta})
	\end{align*}
	This, together with (\ref{eq:LU}), gives us uniformly in $u \in ]0,1]$:
	\begin{align*}
		&\frac{L_n(u)}{\#_1(\ctT_n)}=\frac{L_n'(u)(c_2+1)+O(n^{\frac{1}{2}+\delta})}{|\ctT_n|}+O\left(\frac{\ln(n)}{|\ctT_n|}\right), \text{ that is,}\\
		&L_n(u)\frac{|\ctT_n|}{\#_1(\ctT_n)}=L_n'(u)(c_2+1)+O(n^{\frac{1}{2}+\delta})+O(\ln(n)).
	\end{align*}
	Dividing both sides by $c_2+1$ and using \eqref{eq:approximationofsizeovertype1size}, we get:
	\begin{align*}
		L_n(u) +O(n^{\frac{1}{2}+\delta})&=L_n'(u)+O(n^{\frac{1}{2}+\delta}).
	\end{align*}
	From that, it follows as desired that, uniformly in $u \in ]0,1]$, 
	\[
		|L_n'(u)-L_n(u)| =O(n^{\frac{3}{4}}). \qedhere
	\]
\end{proof}

We now have all the tools to prove Proposition \ref{prop:ConvFlattened}.

\begin{proof}[Proof of Proposition \ref{prop:ConvFlattened}]
	To extend the rooted Gromov--Hausdorff--Prokhorov convergence of the tree $\kT_n$ to the one of the tree $\ctT_n$, we need to control the rooted Hausdorff--Prokhorov distance between both trees (still under the natural embedding of $\kT_n$ into $\ctT_n$). Note that any vertex in $\ctT_n$ is at graph distance at most $1$ from $\kT_n$, so that, almost surely:
	\begin{equation}
		\label{eq:cv hausdorff reduced tree}
		d_H((\kT_n, n^{-1/2}d_{\kT_n}), (\ctT_n, n^{-1/2} d_{\ctT_n})) \underset{n \rightarrow \infty}{\rightarrow} 0,
	\end{equation}
	Therefore, we only need to show that the Prokhorov distance between $m_{\kT_n}$ and $m_{\ctT_n}$ (on the rescaled trees) vanishes with high probability as $n \rightarrow +\infty$. By Lemma \ref{lem:coupling}, there exists a set of trees $\mathcal{E}$ so that 
	\begin{align}\label{eq:treeset}
		\Prb{\ctT_n\in \mathcal{E}} \geq 1 -O(n^{-10})
	\end{align}
	and, on the event $\ctT_n\in \mathcal{E}$, we have
	
	\begin{align}\label{eq:cpldistance}
		\sup_{u \in ]0,1]} |L_n(u)-L_n'(u)|\leq n^{\frac{3}{4}}.
	\end{align}
	Now recall that the graph distance $d_T$ in a tree $T$ is given by: for $v,w \in T$, 
	\begin{align}\label{eq:graphdistance}
		d_{T}(v,w)=h_{T}(v)+h_{T}(w)-2h_{T}(lca(v,w)),
	\end{align}
	where $lca(v,w)$ is the last common ancestor of $v$ and $w$ and $h_{T}(v)$ denotes the height of the vertex $v$ in $T$. Proposition \ref{prop:jointconvergence} gives us that the height process and the contour process of $\kT_n$, both rescaled by $\frac{\kappa_{tree}}{\sqrt{n}}$, jointly converge in distribution to the same Brownian excursion of duration $1$ after time-rescaling. By Skorohod's representation theorem we may assume that this joint convergence holds almost surely. Hence, the continuity of the Brownian excursion together with Equation \eqref{eq:graphdistance} implies that almost surely
	\begin{align}\label{eq:zero}
		\frac{1}{\sqrt{n}}\sup_{v,w}d_{\kT_n}(v,w)\to 0,
	\end{align}
	where the supremum is taken over all pairs of vertices $(v,w)$ in $\kT_n$ whose positions in the right-reversed depth-first-search order differ by at most $n^{3/4}$. 
	
	By the Borel-Cantelli lemma and \eqref{eq:treeset}, almost surely $\ctT_n\notin \mathcal{E}$ only for finitely many values of $n$. This allows us to deduce from \eqref{eq:zero} that, for the coupled $v_n(u),v_n'(u)$ as defined in Definition \ref{def:coupledvertices}, almost surely it holds:
	\begin{align*}
		\frac{1}{\sqrt{n}} \sup_{u \in ]0,1]} d_{\ctT_n}(v_n(u),v'_n(u))\to 0.
	\end{align*}
	Now, \cite{RKSF13}, Corollary $7.5.2$  and \eqref{eq:cv hausdorff reduced tree} imply that
	\begin{align*}
		d_{GHP}\left( \left(\kT_n, \frac{\kappa_{tree}}{\sqrt{n}} d_{\kT_n},m_{\kT_n} \right), \left( \ctT_n, \frac{\kappa_{tree}}{\sqrt{n}} d_{\ctT_n},m_{\ctT_n} \right)\right) \convp 0
	\end{align*}
	as $n \to \infty$.
	This completes our proof.
\end{proof}
\begin{remark}
    Instead of right-reversed depth-first search order we could also work with the usual depth-first search order, but would then have to define the flattened tree in a way that the vertices, which are not of type one are attached on the left instead of on the right.
\end{remark}
\section{From the flattened tree to the original tree}
\label{sec:cvoriginal}

The goal of this section is to prove our main theorem, Theorem \ref{thm:mainthm}, stating the convergence of the multitype \bgw tree towards the Brownian CRT.
To this end, we embed $\cT_n$ and $\ctT_n$ in the same space, on which we can control the distances in the original tree $\cT_n$, compared to the distances in the flattened tree $\ctT_n$.

\subsection{The blow-up operation}

We define in this section the blow-up operation which can be seen as the inverse (in distribution) of the flattening operation of Section \ref{ssec:operations}.

\subsubsection{The $K$-type size-biased tree}

We start by defining a random tree that we will call the $K$-type size-biased tree $\ctTl$ that apart from the root vertex has a marked vertex of type $1$ with height $\ell$. To this end, we define the size-biased random variable $(\txi_1^{\bullet}, \ldots, \txi_K^{\bullet})$ as follows: for all $a_1, \ldots, a_K \in \mathbb{N}_{0}$,
\begin{align*}
	\Prb{(\txi_1^{\bullet}, \ldots, \txi_K^{\bullet})=(a_1, \ldots, a_K)}=a_1 \, \Prb{(\txi_1,\ldots, \txi_K)=(a_1,\ldots, a_K)}.
\end{align*}
Since $\mathbb{E}[\tilde{\xi}_1]=1$ the size-biased random variable is well-defined.
We construct a random tree $\ctTl$ with a marked type $1$ vertex of height $\ell$ recursively as follows. If $\ell=0$ we let $\cT^{\flat(0)}$  be an independent copy of $\ctT$ marked at its root. For $\ell \ge 1$, we start with a type $1$ root vertex that receives offspring according to an independent copy of $(\txi_1^{\bullet}, \ldots, \txi^{\bullet}_K)$. Note that the root  necessarily has at least one child of type $1$, hence we may choose one of them uniformly at random and identify it with the root of $\cT^{\flat(\ell-1)}$. Each of the other type $1$ children gets identified with the root of an independent copy of the  unconditioned tree $\ctT$.

We let $v_0, \ldots, v_\ell$ denote the path from the root $v_0$ of $\ctTl$ to its marked type $1$ vertex~$v_{\ell}$. This path is also called the spine of $\ctTl$ .

\subsubsection{Decorations}
Let $\tau$ be a fixed flat tree, that is, a tree where only vertices of type $1$ can have children. Let $V_1(\tau)$ be the set of vertices in $\tau$ that have type $1$, and $R:=\#_1(\tau)$ be its cardinality. For each vertex $v \in V_1(\tau)$ and $1 \leq i \leq K$, denote by $k_{v}(i)$ the number of children of $v$ of type $i$. \label{defkvi}

Let $\gimel_{\tau,v}$ be the set of $K$-type rooted trees with root of type $1$, $k_{v}(1)$ leaves of type $1$, no internal vertex of type $1$ except the root, and $k_{v}(i)$ vertices of type $i$ for all $2 \leq i \leq K$. A map associating to each vertex $v \in V_1(\tau)$ an element $\alpha(v)$ of $\gimel_{\tau,v}$ is called a decoration, and a pair $(\tau,\alpha)$ where $\alpha$ is a decoration of $\tau$ is called a decorated tree. We now define a random decoration $\alpha_\tau$ of a tree $\tau$ with $R>0$ type-$1$ vertices $v_1,\dots, v_R$  as follows: $\Prb{\alpha_{\tau}(v_1)=T_1, \ldots, \alpha_\tau(v_R)=T_R} \propto \prod_{i=1}^R \Prb{B_{\varnothing}=T_i}$.

\subsubsection{The blow-up operation}
\label{def:blowupoperation}
We define a blow-up operation $\Phi$ on the set of decorated trees. For a decorated $K$-type tree $(\tau,\alpha)$, the blown-up tree $\Phi(\tau,\alpha)$ is constructed in the following way.
Starting from the root vertex $v_1$, let $k_{v_1}(1)$ be its number of children of type $1$. Delete all edges from $v_1$ to its children, delete its children of type $\geq 2$ and identify $v_1$ with the root of $\alpha(v_1)$. Now denote by $u_1,\ldots,u_{k_{v_1}(1)}$ the leaves of $\alpha(v_1)$ in lexicographic order. We identify them (in their order) with the children of type $1$ of $v_1$ in $\tau$. We recursively perform the same procedure on all vertices of type $1$ in $\tau$. The result $\Phi(\tau,\alpha)$ is then a $K$-type tree. For every vertex $v \in \tau$ of type $1$, we let $\Phi(v) \in \Phi(\tau,\alpha)$ be the vertex of type $1$ identified with $v$ by $\Phi$.  Let $W_v$ denote the vertex set of children of $v$, which are not of type $1$. Now we let $\Phi(w)\in\Phi(\tau,\alpha)$ be a vertex according to an arbitrary type-preserving bijective choice between the elements of $W_v$ and the vertices in $\alpha(v)$, which are not of type $1$.

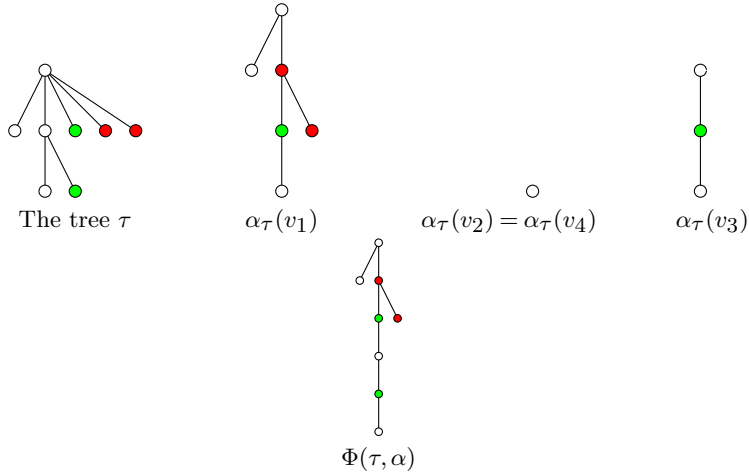
\begin{figure}[!ht]
	\centering
	\begin{tabular}{c c c c c c c}
		
		\begin{tikzpicture}[scale=.8]
			\draw (-0.5,-1) -- (0,0) -- (0,-2) (.5,-1) -- (0,0) -- (1,-1) (1.5,-1) -- (0,0) (.5,-2) -- (0,-1);
			\draw[fill=white] (0,0) circle(.1);
			\draw[fill=white] (0,-1) circle(.1);
			\draw[fill=white] (-.5,-1) circle(.1);
			\draw[fill=white] (0,-2) circle(.1);
			\draw[fill=green] (.5,-1) circle(.1);
			\draw[fill=green] (.5,-2) circle(.1);
			\draw[fill=red] (1.5,-1) circle(.1);
			\draw[fill=red] (1,-1) circle(.1);
		\end{tikzpicture}
		&  
		\begin{tikzpicture}
			\draw[white] (0,0) -- (0.5,0);
		\end{tikzpicture}
		&
		\begin{tikzpicture}[scale=.8]
			\draw (-0.5,-1) -- (0,0) -- (0,-2) --(0,-3) (0,-1)--(0.5,-2);
			\draw[fill=white] (0,0) circle(.1);
			\draw[fill=red] (0,-1) circle(.1);
			\draw[fill=white] (-.5,-1) circle(.1);
			\draw[fill=green] (0,-2) circle(.1);
			\draw[fill=white] (0,-3) circle(.1);
			\draw[fill=red] (0.5,-2) circle(.1);
		\end{tikzpicture} 
		&  
		\begin{tikzpicture}
			\draw[white] (0,0) -- (0.5,0);
		\end{tikzpicture}
		&
		\begin{tikzpicture}[scale=.8]
			\draw[fill=white] (0,0) circle(.1);
		\end{tikzpicture} 
		&  
		
		\begin{tikzpicture}[scale=.8]
			\draw (1.5,0) --(1.5,-2);
			\draw[fill=white] (1.5,0) circle(.1);
			\draw[fill=green] (1.5,-1) circle(.1);
			\draw[fill=white] (1.5,-2) circle(.1);
			\draw[white] (1.5,0) -- (2,0);
		\end{tikzpicture} 
		\\
		The tree $\tau$   &  
		\begin{tikzpicture}
			\draw[white] (0,0) -- (0.5,0);
		\end{tikzpicture} & 
		$\alpha_\tau(v_1)$  &  
		\begin{tikzpicture}
			\draw[white] (0,0) -- (0.5,0);
		\end{tikzpicture} & 
		$\alpha_\tau(v_2)=\alpha_\tau(v_4)$
		\begin{tikzpicture}
			\draw[white] (0,0) -- (.5,0);
		\end{tikzpicture} & 
		$\alpha_\tau(v_3)$ 
	\end{tabular}

	\begin{tabular}{c}
		
		\begin{tikzpicture}[scale=.5]
			\draw (-0.5,-1) -- (0,0) -- (0,-2) --(0,-5) (0,-1)--(0.5,-2);
			\draw[fill=white] (0,0) circle(.1);
			\draw[fill=red] (0,-1) circle(.1);
			\draw[fill=white] (-.5,-1) circle(.1);
			\draw[fill=green] (0,-2) circle(.1);
			\draw[fill=white] (0,-3) circle(.1);
			\draw[fill=red] (0.5,-2) circle(.1);
			\draw[fill=green] (0,-4) circle(.1);
			\draw[fill=white] (0,-5) circle(.1);
		\end{tikzpicture}\\
		$\Phi(\tau,\alpha)$
	\end{tabular}
	\caption{Top left: a flat tree $\tau$. Its vertices of type $1$ are drawn in white, those of type $2$ in green, and those of type $3$ in red. Top right: a decoration $\alpha$ of $\tau$. Bottom: the blown-up tree $\Phi(\tau,\alpha)$.}
	\label{fig:blown up tree}
	
\end{figure}
The main observation is that the tree  $\Phi(\ctT_n,\alpha_{\ctT_n})$ is distributed as $\cT_n$. For convenience, we will often write $\Phi(\ctT_n)$ instead of $\Phi(\ctT_n,\alpha_{\ctT_n})$ and $\alpha_n$ instead of $\alpha_{\ctT_n}$.

\begin{lemma}
	For a finite decorated rooted tree $(\tau,\alpha_{\tau})$ with a marked vertex $u$ of type $1$ and height $\ell$, it holds that
	\begin{align*}
		\Prb{(\ctTl,\alpha_{\ctTl})=(\tau,\alpha_{\tau})}= \Prb{(\ctT,\alpha_{\ctT})=(\hat{\tau},\alpha_{\hat{\tau}})},
	\end{align*}
	where $\hat{\tau}$ is the unmarked tree obtained from $\tau$ when dropping the marking on the spine, and $\alpha_{\hat{\tau}}$ is the same decoration as $\alpha_\tau$.
\end{lemma}
\begin{proof}
	Let $(\tau,v)$ be a flat $K$-type rooted tree $\tau$ together with a marked type $1$ vertex $v$ at height $\ell$.
	We first observe
	\begin{align}
		\label{eq:firststep}
	\Prb{\ctTl=(\tau, v)} = \Prb{\ctT= \tau}.
	\end{align}
	To see this, note that the only difference between $\ctT$ and $\ctTl$ is that in $\ctTl$ the vertices $v_0, \ldots, v_{\ell-1}$ of the spine receive offspring according to independent copies of the  size-biased random variable $(\txi_1^{\bullet}, \ldots, \txi_K^{\bullet})$ instead of $(\txi_1, \ldots, \txi_K)$. For each $0 \le i < \ell$ the vertex $v_{i+1}$ gets chosen uniformly at random among the type $1$ children of $v_i$. This way, the spine and  hence the marked vertex $v_\ell$ is determined. Note that for $a_1, \ldots, a_K \in \mathbb{N}_0$, the probability for $v_i$ to have precisely $a_j$ children of type $j$ for all $1 \le j \le K$, and for the uniform choice of $v_{i+1}$ among the $a_1$ type $1$ children to equal a fixed outcome is given by
	\[
	\frac{1}{a_1} \Prb{(\txi_1^{\bullet}, \ldots, \txi_K^{\bullet})=(a_1, \ldots, a_K)}= \Prb{(\txi_1,\ldots, \txi_K)=(a_1,\ldots, a_K)}.
	\]
	Therefore, the result of forgetting which vertex of $\ctTl$ is marked is distributed like $\ctT$, and Equation~\eqref{eq:firststep} follows. 
	
	Since the distribution of the decoration only depends on the underlying tree, the proof is complete.
\end{proof}

We now use this result on the conditioned tree. For every $\varepsilon,c>0$, $n \geq 1$, we let $E_{n,\varepsilon,c}$ be the event that there exists a vertex $v \in \ctT_n$ with height $h_{\ctT_n}(v)\geq \ln(n)^4$ in $\ctT_n$, and height $h_{\Phi(\ctT_n)}(\Phi(v))\notin [c(1-\varepsilon)h_{\ctT_n}(v),c(1+\varepsilon)h_{\ctT_n}(v)]$ in $\Phi(\ctT_n)$.

\begin{lemma}
	\label{lem:rareevent}
	There exist constants $c_{\varsigma}$ and $A',B>0$ such that, for all $n \geq 1$ and $0<\varepsilon<1$, we have that $\Prb{E_{n,\varepsilon,c_\varsigma}}<A'n^{7/2-B\varepsilon^2\ln(n)}$.
\end{lemma}

\begin{proof}
	Fix $c,\varepsilon>0$ and $0<\delta<1/2$. For $\ell,n \geq 1$, let $\mathcal{E}_{\ell,n,c}$ be the set of flat marked decorated $K$-trees $(\tau,\alpha_{\tau})$ such that $\#_{\blambda}(\tau)=n$, with a marked vertex $w_{\tau}$ of type $1$ at height $\ell$, such that the height of $\Phi(w_{\tau})$ in $\Phi((\tau,\alpha_{\tau}))$ is not in the interval $[c(1-\varepsilon)\ell,c(1+\varepsilon)\ell]$. Recall that for $\tau$ a marked tree, $\hat{\tau}$ denotes the unmarked version of it.
	
	Let $m:=\#_1(\ctT_n)$ be the (random) number of vertices of type $1$ in the tree $\ctT_n$. Setting $g_n := \lceil \frac{n}{c_1}+n^{\frac{1}{2}+\delta} \rceil$, we have
	\begin{align*}
		\Prb{E_{n,\varepsilon,c}}\leq\Prb{E_{n,\varepsilon,c}, m\leq g_n}+\Prb{m > g_n}.
    \end{align*}
    We can bound
    \begin{align*}
        \Prb{E_{n,\varepsilon,c}, m\leq g_n} &\leq\sum\limits_{\ell=\lfloor \ln^4(n)\rfloor}^{ g_n+1}\sum\limits_{(\tau,\alpha_{\tau})\in\mathcal{E}_{\ell,n,c} }\Prb{(\ctT_n,\alpha_{\ctT_n})=(\hat{\tau},\alpha_{\hat{\tau}})}\\
		&=\sum\limits_{\ell=\lfloor \ln^4(n)\rfloor}^{g_n+1}\sum\limits_{(\tau,\alpha_\tau)\in\mathcal{E}_{\ell,n,c} }\Prb{(\ctT,\alpha_{\ctT})=(\hat{\tau},\alpha_{\hat{\tau}})|\#_{\blambda}\ctT=n}\\
		&=\sum\limits_{\ell=\lfloor \ln^4(n)\rfloor}^{ g_n+1}\sum\limits_{(\tau,\alpha_{\tau})\in\mathcal{E}_{\ell,n,c} }\frac{\Prb{(\ctT,\alpha_{\ctT})=(\hat{\tau},\alpha_{\hat{\tau}}),\#_{\blambda}\ctT=n}}{\Prb{\#_{\blambda}\ctT=n}}\\
		&\leq \sum\limits_{\ell=\lfloor \ln^4(n)\rfloor}^{g_n+1}\sum\limits_{(\tau,\alpha_{\tau})\in\mathcal{E}_{\ell,n,c} }\frac{\Prb{(\ctT,\alpha_{\ctT})=(\hat{\tau},\alpha_{\hat{\tau}})}}{\Prb{\#_{\blambda}\ctT=n}}\\
		&=O(n^{\frac{3}{2}}) \sum\limits_{\ell=\lfloor \ln^4(n)\rfloor}^{ g_n+1}\sum\limits_{(\tau,\alpha_{\tau})\in\mathcal{E}_{\ell,n,c} }\Prb{(\ctTl,\alpha_{\ctTl})=(\tau,\alpha_{\tau})}\\
		&=O(n^{\frac{3}{2}}) \sum\limits_{\ell =\lfloor \ln^4(n)\rfloor}^{g_n+1}\Prb{(\ctTl,\alpha_{\ctTl})\in \mathcal{E}_{\ell,n,c}},
	\end{align*}
	\noindent where the $O(n^{\frac{3}{2}})$ does not depend on $\varepsilon$.
	
    Now recall the construction of $\ctTl$. In particular, $v_1$ is a uniform child of type $1$ of the root. Let $\varsigma$ be distributed like the height of $\Phi(v_1)$ in the blown-up tree, and $(\varsigma^{(i)})_{i \geq 1}$ be i.i.d. copies of $\varsigma$.
	
	Note that since $\txi$ has finite exponential moments, $\varsigma$ has finite exponential moments. Let $c_{\varsigma}$ denote the expected value of $\varsigma$.
	Now we may use Lemma \ref{DeviationLemma} (with $\varrho=\varepsilon c_{\varsigma} / (2D)$ for $D$ as in the notation of Lemma \ref{DeviationLemma}) to deduce that there exist $A,B,\varepsilon_0>0$ such that for all small enough $0<\varepsilon<\varepsilon_0$ we have:
	\begin{align*}
		\Prb{\left|\sum\limits_{i=1}^{\ell}\varsigma^{(i)}-\ell c_{\varsigma}\right|\geq c_{\varsigma}\varepsilon\ell}\leq A\exp\left(-B \ell\varepsilon^2 \right).
	\end{align*}
	By potentially replacing $B$ by a smaller constant we may  without loss of generality assume that $\varepsilon_0=1$. On the other hand, we have $\Prb{m > g_n} \leq A''\exp(-B''n^{\delta})$, for some $A'',B''>0$ by Proposition \ref{prop:concentration} (i). 
	This implies that, for $\tilde{A}>0$:
	\begin{align*}
		\Prb{E_{n,\varepsilon,c_\varsigma}}
		&\leq O(n^{\frac{5}{2}})\sum\limits_{\ell=\lfloor \ln^4(n)\rfloor}^{g_n+1}\tilde{A}\exp\left(-B \ell \varepsilon^2\right)+ A''\exp(-B''n^{\delta}).
	\end{align*}
	Using $n+n^{1/2+\delta}+1\leq 3n$ for $n\geq 1$ there exist constants $A'>0$ (determined by the value of $c_1$) and $B',B''' >0$, such that
	\begin{align*}
		\Prb{E_{n,\varepsilon,c_\varsigma}}&\leq O(n^{\frac{5}{2}}) 3n A' \exp\left(-B'\varepsilon^2 \ln^4(n)\right)+A''\exp(-B''n^{\delta})\\
		&=O(n^{\frac{7}{2}}) n^{-B'\varepsilon^2 \ln(n)}+ A''\exp(-B''n^{\delta})\\
		&\leq O(n^{\frac{7}{2}})n^{-B'''\varepsilon^2 \ln(n)}.
	\end{align*}
	As this bound tends to zero as $n \to \infty$, we have proven the statement of the present lemma for vertices of type $1$. By Lemma~\ref{lem:largestoutdegree} with high probability any vertex of $\ctT_n$ is at distance at most $O(\log n)$ from its youngest type $1$ ancestor in $\cT_n$. Hence the proof is complete for vertices with arbitrary type.
\end{proof}

This finally allows us to control the height of a vertex $\Phi(v)$ with respect to the height of $v$ in the original flattened tree:

\begin{lemma}\label{lem:distance}

	For every $0< \varepsilon<1$ we have with high probability that for every vertex $v\in \ctT_n$ 
	\begin{align*}
		h_{\Phi(\ctT_n)}(\Phi(v))\in [c_{\varsigma}(1-\varepsilon)h_{\ctT_n}(v)+O(\ln^5(n)),c_{\varsigma}(1+\varepsilon)h_{\ctT_n}(v)+O(\ln^5(n))]
	\end{align*}
\end{lemma}
	\begin{proof}
		By Lemma \ref{lem:rareevent}, with high probability, for any $v \in \ctT_n$ such that $h_{\ctT_n}(v) \geq \ln^4(n)$, the statement holds. 
		Now, observe that, by Corollary \ref{cor:sizeoflargestblob}, there exists $C>0$ such that, with high probability, for all $v \in \ctT_n$, $h_{\Phi(\ctT_n)}(\Phi(v)) \leq C \ln(n) h_{\ctT_n}(v)$. In particular, on that event, for all $v \in \ctT_n$ such that $h_{\ctT_n}(v) \leq \ln^4(n)$, $h_{\Phi(\ctT_n)}(\Phi(v)) \leq C \ln^5(n)$. The result follows. 
	\end{proof}

\subsection{Proof of Theorem \ref{thm:mainthm}}

For convenience, for any tree $T$ and any $a>0$, we set: $aT := (T, a d_T, m_T)$.
We can now finally prove Theorem \ref{thm:mainthm}, using the fact that the blown-up tree is distributed like the original tree.

\begin{proposition}
	\label{prop:cvhausdorff1}    
	As $n \rightarrow\infty$, 
	\begin{align*}
		d_{GHP}\left( n^{-1/2} c_{\varsigma} \ctT_n, n^{-1/2} \Phi(\ctT_n) \right) \convp 0.
	\end{align*}
\end{proposition}

\begin{proof}
	We have for any $u,v \in \ctT_n$ and the corresponding $\Phi(u), \Phi(v)\in \Phi((\ctT_n,\alpha_n))$.
	\begin{align*}
		d_{\ctT_n}(u,v)=h_{\ctT_n}(u)+h_{\ctT_n}(v)-2h_{\ctT_n}(lca_{\ctT_n}(u,v))
	\end{align*}
	The height of a vertex $\Phi(w)$, where $w=lca_{\ctT_n}(u,v)$ differs from the one of $w'=lca_{\Phi((\ctT_n,\alpha_n))}(\Phi(u),\Phi(v))$ with high probability at most $O(\ln(n))$, i.e.:
	$$\sup_{u,v \in \ctT_n}|h_{\Phi(\ctT_n)}(\Phi(w))-h_{\Phi(\ctT_n)}(w')|=O(\ln(n)).$$
    Indeed, with the same notation, for all $u,v$, the vertex $w'$ is an element of the image of $\alpha_n(w)$ in $\Phi((\ctT_n,\alpha_n))$. Hence, $d_{\Phi(\ctT_n)}(\Phi(w),w')\leq |\alpha_n(w)|=O(\ln n)$ in probability by Corollary \ref{cor:sizeoflargestblob}, uniformly in $u,v$.
	Therefore with high probability we have 
	\begin{align*}
		d_{\Phi(\ctT_n)}(\Phi(u),\Phi(v))&=h_{\Phi(\ctT_n)}(\Phi(u))+h_{\Phi(\ctT_n)}(\Phi(v))\\
		&-2h_{\Phi(\ctT_n)}(\Phi(lca_{\ctT_n}(u,v)))+O(\ln(n)).
	\end{align*}
	By Lemma \ref{lem:distance} we thus have
	\begin{align*}
		\sup_{u,v\in \ctT_n}\frac{|c_{\varsigma}d_{\ctT_n}(u,v)-d_{\Phi(\ctT_n)}(\Phi(u),\Phi(v))|}{\sqrt{n}}\convp 0.
	\end{align*}
	This implies that the rooted Gromov--Hausdorff distance between  $n^{-1/2} c_{\varsigma} \ctT_n$ and $n^{-1/2}\Phi(\ctT_n)$ tends to zero in probability. Since the two trees share the same set of vertices, it follows that the rooted Gromov--Hausdorff--Prokhorov distance with respect to the uniform measure on both trees tends to zero in probability as well. This completes the proof. 
\end{proof}

This allows us to compare the trees $(\ctT_n,d_{\ctT_n},m_{\ctT_n})$ and $(\cT_n,d_{\cT_n},m_{\cT_n})$.

\begin{proposition}
	\label{prop:cvhausdorff}
	As $n \rightarrow\infty$, we have the following convergence in probability:
	\begin{align*}
		d_{GHP}\left( n^{-1/2} c_{\varsigma} \ctT_n, n^{-1/2} \cT_n \right) \rightarrow 0.
	\end{align*}
\end{proposition}

\begin{proof}
	
	Using that $\Phi(\ctT_n)$ is distributed as $\cT_n$, we get from Proposition \ref{prop:cvhausdorff1} that, in distribution:
	\begin{equation*}
		d_{GHP}\left( n^{-1/2} c_{\varsigma} \ctT_n, n^{-1/2} \cT_n \right) = d_{GHP}\left( n^{-1/2} c_{\varsigma} (\Phi(\ctT_n))^{\flat}, n^{-1/2} \Phi(\ctT_n) \right),
	\end{equation*}
	where we recall that $T^{\flat}$ stands for the flattened version of $T$.
	Now, observe that, almost surely,
	\begin{equation*}
		(\Phi(\ctT_n))^{\flat} = \ctT_n.
	\end{equation*}
	The result follows.
\end{proof}

We finally obtain a proof of our main theorem.

\begin{proof}[Proof of Theorem \ref{thm:mainthm}]
	By Proposition \ref{prop:ConvFlattened}, we have that, for 
	\begin{equation}
		\kappa_{tree}:=\frac{1}{2} \sqrt{\mathbb{V}[\txi_1] \, \sum\limits_{i=1}^K \lambda_i \mathbb{E}[\txi_i]},
	\end{equation}
	\noindent the following convergence holds in distribution for the rooted Gromov--Hausdorff--Prokhorov distance:
	\begin{align}
		\label{eq:convergence in the last proof}
		(\ctT_n,\kappa_{tree}n^{-\frac{1}{2}}d_{\ctT_n},m_{\ctT_n})\convd (\mathcal{T}_{\be},d_{\mathcal{T}_{\be}},m_{\mathcal{T}_{\be}}).
	\end{align}
	Now, by Proposition \ref{prop:cvhausdorff} it holds:  
	\begin{align*}
		d_{GHP}\left( c_{\varsigma} n^{-1/2} \ctT_n, n^{-1/2} \cT_n \right) \rightarrow 0,
	\end{align*}
	where $c_{\varsigma}>0$ is the expected height of the vertex $\Phi(v_1)$, see Lemma \ref{lem:rareevent}. This entails
	\begin{align*}
		d_{GHP}\left( n^{-1/2} \kappa_{tree} \ctT_n, n^{-1/2} \kappa_{tree} \frac{1}{c_{\varsigma}}\cT_n \right) \rightarrow 0,
	\end{align*}
	\noindent and the result follows from \eqref{eq:convergence in the last proof}. In particular, the scaling constant in the convergence of Theorem \ref{thm:mainthm} is given by
	\begin{align*}
		c_{\mathrm{scal}} = \frac{1}{c_{\varsigma}} \kappa_{tree}=\frac{1}{2 c_{\varsigma}}
		\sqrt{\mathbb{V}[\txi_1]\sum\limits_{i=1}^K \lambda_i \mathbb{E}\left[\tilde{\xi}_i\right]}=\frac{1}{2 c_{\varsigma}}
		\sqrt{c_1\mathbb{V}[\txi_1]}.
	\end{align*}

The last thing to prove is that we have the equality
\begin{align*}
c_{\mathrm{scal}}=\frac{\sigma}{2} \sqrt{\sum_{i=1}^K \lambda_i a_i},
\end{align*}
with $\sigma$ given in Equation~\eqref{eq:sigma} and $(a_1, \ldots, a_K)$ is the $1$-left eigenvector of the mean matrix $M$ with positive coordinates and such that $\sum_{i=1}^K a_i = 1$. To this end, first observe that, by construction, the constant $c_{\varsigma}$ is independent of $\blambda$. Hence, recalling that $\mathbb{E}\left[\tilde{\xi}_1\right]=1$ (Lemma \ref{lem:meanexpomoments}) and taking $\blambda:=(1,0,\ldots,0)$, Miermont's result \cite{Mie08}, Theorem 2 implies:
\begin{align}
	\label{eq:mimi}
    \frac{1}{2c_{\varsigma}} \sqrt{\mathbb{V}[\txi_1]} = \frac{\sigma \sqrt{a_1}}{2}.
\end{align}

On the other hand, \cite{Ste18}, Proposition 2.1 (ii) (see also \cite{Mie08}, Proof of Proposition 4 (ii)) shows that, for all $i \in \{1,\ldots,K\}$, $\mathbb{E}[\txi_i]=\frac{a_i}{a_1}$. Hence, we get
\begin{align*}
    c_{\mathrm{scal}} &= \frac{1}{2 c_{\varsigma}}
		\sqrt{\mathbb{V}[\txi_1]\sum\limits_{i=1}^K \lambda_i \mathbb{E}\left[\tilde{\xi}_i\right]}\\
        &= \frac{\sigma \sqrt{a_1}}{2} \sqrt{\sum_{i=1}^K \lambda_i \frac{a_i}{a_1}}\\
        &= \frac{\sigma}{2} \sqrt{\sum_{i=1}^K \lambda_i a_i}.
\end{align*}

\end{proof}

\subsection{Proof of Theorem~\ref{te:bound}}

We are now ready to finalize the proof of Theorem~\ref{te:bound}.
\begin{proof}[Proof of Theorem~\ref{te:bound}]
		It suffices to verify existence of constants $C,c>0$ such that ~\eqref{eq:tebound} holds uniformly for all $x>\sqrt{n}$. 
This is because we may replace $C$ by a larger positive constant and $c$ by a smaller positive constant (so that $C \exp(-c) > 1$), ensuring that~\eqref{eq:tebound} is trivially satisfied for $x \le \sqrt{n}$.
		
		Let $\mathcal{E}_{\ell,n}$ denote the collection of all flat marked decorated $K$-type trees $(\tau, \alpha_\tau)$ with $\#_{\bm{\lambda}} \tau = n$ with a marked vertex $w_\tau$ at height $\ell$ such that there exists at least one vertex $v$ in $\tau$ with the following two properties:
		\begin{itemize}
			\item[(i)] $v =w_\tau$ or $v$ is a child of $w_\tau$ with arbitrary type,
			\item[(ii)] the height of $\Phi(v)$ in  $\Phi((\tau, \alpha_\tau))$ is larger than $x$.
		\end{itemize}		
			With foresight, set
		\begin{align}
			\label{eq:choiceofalpha}
		\alpha = \frac{1}{4 \sum_{j=1}^K \Ex{\tilde{\xi}^\bullet_j}}.
		\end{align}
This way,
		\begin{align}
			\label{eq:s0}
			\Prb{\mathrm{H}(\cT_n)>x} \le \Prb{\mathrm{H}(\ctT_n) > \alpha x} +\sum_{0 \le \ell \le \alpha x}\Prb{(\ctT_n,\alpha_{\ctT_n})\in \mathcal{E}_{\ell,n}}.
		\end{align}
			By Proposition~\ref{pro:goup} it follows that there exist constants $C_1, c_1 >0$
		such that for all $n$ 
		\begin{align}
			\label{eq:s1}
					\Prb{\mathrm{H}(\ctT_n)\geq x\alpha }\leq C_1\exp(-c_1 \alpha^2 x^2/n) + C_1\exp(-c_1\alpha x).
		\end{align}
	Arguing analogously as in the proof of Lemma~\ref{lem:rareevent}, we have
		\begin{align*}
			\sum_{ 0 \le \ell \le \alpha x} \Prb{(\ctT_n,\alpha_{\ctT_n})\in \mathcal{E}_{\ell,n}} &\le  O(n^{\frac{3}{2}}) \sum_{0 \le \ell \le \alpha x}\Prb{(\ctTl,\alpha_{\ctTl})\in \mathcal{E}_{\ell,n}}.		
		\end{align*}
	For any $(\tau, \alpha_\tau) \in \mathcal{E}_{\ell,n}$ any vertex $v$ of $\tau$ satisfying (i) and (ii) above we have that the height  of $\Phi(v)$ in $\Phi((\tau, \alpha_\tau))$ is bounded by the sum of the number of children of all its ancestors. Hence
	\begin{align*}
		\Prb{(\ctTl,\alpha_{\ctTl})\in \mathcal{E}_{\ell,n}} &\le \Prb{Y + \sum_{i=1}^\ell X_i >x}\\
		 &\le \Prb{Y > x/2} + \Prb{\sum_{i=1}^\ell X_i >x/2}
	\end{align*}
	with $Y := \sum_{j=1}^K \tilde{\xi}_j$ and $(X_i)_{i \ge 1}$ denoting independent copies of the random variable $X := \sum_{j=1}^K \tilde{\xi}^\bullet_j$. Using the assumption that $x>\sqrt{n}$ and the fact that $Y$ has finite exponential moments, it follows that there exist $C_2, c_2>0$ with
	\begin{align}
		\label{eq:s2}
	O(n^{3/2}) \Prb{Y > x/2} \le C_2 \exp(-c_2 x).
	\end{align}
	Using Lemma~\ref{DeviationLemma} and $\ell \le \alpha x$ we obtain that there exist $r,D>0$ such that for all $0 \le \varrho \le r$
	\begin{align*}
		O(n^{3/2}) \Prb{\sum_{i=1}^\ell X_i >x/2} &= O(n^{3/2}) \exp(D\ell \varrho^2-\varrho (x/2 - \ell \mathbb{E}[X])) \\
		&\le  O(n^{3/2}) \exp( x\varrho (D \varrho \alpha + \alpha\mathbb{E}[X] - 1/2) ).
	\end{align*}
	By~\eqref{eq:choiceofalpha} we may select $\varrho>0$ sufficiently small so that $D \varrho \alpha + \alpha\mathbb{E}[X] - 1/2 <0$. Using the assumption that $x>\sqrt{n}$, it then follows that
	\begin{align}
		\label{eq:s3}
		O(n^{3/2}) \Prb{\sum_{i=1}^\ell X_i >x/2} &\le C_3 \exp(-c_3 x)
	\end{align}
	for constants $C_3, c_3>0$ that do not depend on $n$ or $x$. Combining the inequalities~\eqref{eq:s0},~\eqref{eq:s1},~\eqref{eq:s2}, and~\eqref{eq:s3} it follows that there exist $C,c>0$ with
	\begin{align*}
				\Prb{\mathrm{H}(\cT_n) > x } \le C \exp(-cx^2 /n ) + C\exp(-cx).
	\end{align*}
	This completes the proof.
\end{proof}

\section{The reducible case}
\label{sec:reducible}

\subsection{Preparations}

The aim of this section is to prove Theorem~\ref{thm:reducible}. We hence assume the conditions on $\bm{\zeta}$ stated there are satisfied.

We first verify a statement analogous to Lemma~\ref{lem:meanexpomoments}. All auxiliary random variables are defined analogously as in the irreducible case. We let $\rho(\cdot)$ denote the spectral radius.

\begin{lemma}
	\label{lem:aa}
	We have $\mathbb{E}[\txi_1]=1$, $(\txi_1,\ldots, \txi_{K+K'})$ has finite exponential moments, and $\Prb{\txi_1=0}>0$.
\end{lemma}
\begin{proof}
	 Consider the $(K+K')$-type branching process where any vertex has the same number of children with type $\ge 2$ as in $\cT$, but no children of type $1$. The mean matrix of this process is given by $(m_{i,j}^*)_{1 \le i,j \le K + K'}$ with $m_{i,j}^* = \mathbb{E}[\xi_j^{(i)}]$ for $j > 1$ and $m_{i,1}^* = 0$. 
	
	This process is subcritical. Indeed, since $\rho(M')<1$ it suffices to show that $\rho(M^*)<1$ for  $M^* := (m_{i,j}^*)_{1 \le i,j \le K }$. To see this, note that $M^*$ is obtained from $M$ by setting the coefficients of the first column to zero, and hence it follows from the monotonicity of the spectral radius for matrices with nonnegative coefficients (which is a consequence of Gelfand's formula) that 
	\[
		\rho(M^*) \le \rho(M) = 1.
	\]
	Hence we need to rule out the case $\rho(M^*)=1$. If $\rho(M^*)=1$, then by the Perron-Frobenius theorem for possibly reducible nonnegative matrices there exists a $1$-eigenvector $v$ of $M^*$ with nonnegative coefficients. Coefficient-wise, $Mv$ is larger than or equal to $M^*v = v$.	If $Mv=v$, then by the Perron-Frobenius theorem all coefficients of $v$ need to be positive, but since at least one coefficient of the first column of $M$ is positive this implies $v = Mv \ne M^*v = v$, a clear contradiction. Thus, we must have $Mv \ne v$. Letting $I$ denote the unit matrix, it follows that $(M-I)v$ has nonnegative coefficients, and at least one nonzero coordinate.  Letting $w$ denote a left $1$-eigenvector of $M$ with positive coefficients, it follows that $w(M-I)v>0$ (because $w$ has positive coefficients) and at the same time $w(M-I)v = 0v=0$ (because $w$ is a left $1$-eigenvector of $M$), a clear contradiction. Thus, $\rho(M^*)<1$.
	
	It follows directly from the implicit function theorem that the total population of any subcritical multitype branching process has finite exponential moments, if its offspring distribution admits finite exponential moments.  Since $\rho(M^*)<1$, this implies that $(\txi_1,\ldots, \txi_{K+K'})$ has finite exponential moments. 
	
	Recall that by our assumptions at least one type in $[K]$ has the property, that the probability of having exactly one child is less than $1$. This entails that $\Prb{\txi_1\ge 2}>0$. Furthermore, we have $\Prb{\txi_1=0}>0$, because otherwise $|\cT|=\infty$ almost surely, a contradiction since $\cT$ is critical and hence almost surely finite. Since the type $1$ population of $\cT$ is distributed like the total population of a $1$-type branching process with offspring mechanism $\tilde{\xi}_1$, the fact that the total population of $\cT$ is almost surely finite but has infinite expectation readily implies that the same must hold for the type $1$ subpopulation (because $(\tilde{\xi}_1, \ldots, \tilde{\xi}_{K+K'})$ has finite exponential moments), which by well-known results for monotype branching processes implies $\mathbb{E}[\txi_1] = 1$.
\end{proof}

\begin{lemma}
	\label{lem:bb}
	There exists a constant $c_{\bm \lambda}>0$ and integers $d \ge 1$ and $0 \le q \le d-1$ such that
	\[
		\Prb{ \#_{\bm{\lambda}} \cT = n} \sim c_{\bm \lambda} n^{-3/2}
	\]
	as $n \equiv q \mod d$ tends to infinity, and $\Prb{ \#_{\bm{\lambda}} \cT = n} =0$ for $n \not\equiv q \mod d$.
\end{lemma}
\begin{proof}
	Let 
	\[
	f(x_1, x_2, \ldots, x_{K+K'}) := \mathbb{E}[x_1^{\txi_1}  \cdots x_{K+K'}^{\txi_{K+K'}} ].
	\]
	This way, $Z(x) := \mathbb{E}[x^{\#_{\bm{\lambda}} \cT }]$ satisfies
	\begin{align*}
		 Z(x) = x^{\lambda_1} f(Z(x), x^{\lambda_2}, \ldots, x^{\lambda_{K+K'}}) = E(x, Z(x))
	\end{align*}
	for $E(x,y) := x^{\lambda_1} f(y, x^{\lambda_2}, \ldots, x^{\lambda_{K+K'}})$.  
	
	The fact that $\cT$ is almost surely finite and that the type $1$ subpopulation is distributed like the total population of a $\txi_1$-monotype branching process imply that $Z(x)$ has radius of convergence $1$ and $Z(1)=1$.
	
	This allows us to apply~\cite{MR2240769}, Theorem 28, yielding the statement for integers $d \ge 1$, $0 \le q \le d-1$ and a constant
	\[
		c_{\bm \lambda} = d \sqrt{\frac{ \frac{\partial E}{\partial x}(1,1)}{ 2 \pi \frac{\partial^2 E}{\partial y^2}(1,1) }  } >0 .
	\]
\end{proof}

\subsection{Proof of Theorem~\ref{thm:reducible}}

With Lemma~\ref{lem:aa} and Lemma~\ref{lem:bb} at hand, the proof of Theorem~\ref{thm:reducible} is now fully analogous to the proofs in the irreducible case. We obtain the scaling constant
\begin{align*}
	c_{\mathrm{scal}}':=\frac{\sqrt{\mathbb{V}[\txi_1] }}{c_{\varsigma}}	
	\frac{\sqrt{\left(\sum\limits_{i=1}^{K+K'} \lambda_i \mathbb{E}\left[\tilde{\xi}_i\right]\right)}}{2}.
\end{align*}
It remains to simplify this expression. Here, $c_{\varsigma}>0$ denotes for a uniformly chosen child of the root in a size-biased version of the reduced tree, the expected height of its corresponding vertex after the blow-up operation.
This constant is the same as for the irreducible process where each vertex has the same random number of children with type in $\{1, \ldots, K\}$ and no children of type larger than $K$. Therefore, as in Equation~\eqref{eq:mimi}
\begin{align}
	\label{eq:redmimi}
	\frac{\sqrt{\mathbb{V}[\txi_1] }}{c_{\varsigma}}	  = \sigma \sqrt{a_1}
\end{align}
with $\mathbf{a}:=(a_1,\ldots,a_K)$ the unique $1$-left eigenvector with positive coordinates of the matrix $M$ such that $\sum_{i=1}^K a_i=1$, and 
\begin{align*}
	\sigma^2=\sum_{i,j,k=1}^K a_i b_j b_k Q^{(i)}_{j,k},
\end{align*}
with $\mathbf{b}:=(b_1,\ldots,b_K)$ the unique $1$-right eigenvector of $M$ with positive coordinates such that $\sum_{i=1}^K a_i b_i=1$, and
\begin{align*}
	Q^{(i)}_{j,j} = \sum_{\mathbf{z} \in \N_0^K} \mu^{(i)}(\mathbf{z}) z_j(z_j-1), \text{ and } Q^{(i)}_{j,k}=\sum_{\mathbf{z} \in \N_0^K}\mu^{(i)}(\mathbf{z}) z_j z_k \text{ for }j \neq k.
\end{align*}

By the Perron-Frobenius theorem and its block upper triangular form, the mean matrix of $\bm{\zeta}$ has a left $1$-eigenvector  with nonnegative coordinates. Due to the block upper triangular form of the mean matrix of $\bm{\zeta}$, it follows that it is of the form $(\bm{a}, \bm{a'})$ with $\bm{a}$ a left $1$-eigenvector of $M$ and 
\begin{align}
	\bm{a'} = \bm{a} S (I -M')^{-1}.
\end{align}
Here $I$ denotes the $K' \times K'$ unit matrix and $I-M'$ is invertible since $\rho(M')<1$. We will also use the notation $I$ for referring to the unit matrix of a different dimension whenever it is clear from the context what the dimension is. Furthermore, we use the notation $(\bm{a}, \bm{a'})$ to denote the concatenation of the coordinates of $\bm{a}$ and $\bm{a'}$, that is, $(\bm{a}, \bm{a'})$ refers to the $(K+K')$-dimensional row vector
\[
(\bm{a}, \bm{a'}) = (a_1, \ldots, a_{K+K'}).
\]

For all $1 \le i,j \le K+K'$ we define $\tilde{\xi}_{i,j}$ analogously to $\tilde{\xi}_j$ as the number of nonroot vertices with type $j$ in a branching process started with a single individual of type $i$ that has almost the same branching mechanism as $\cT$, with the only difference being that nonroot vertices with type $1$ are infertile. This way,  $\tilde{\xi}_j$ is equal in distribution to $\tilde{\xi}_{1,j}$. With $A = (a_{i,j})_{1 \le i,j \le K+K'}$ denoting the mean matrix of the offspring distribution $\bm{\zeta}$, we have
\begin{align}
	\label{eq:f1}
	\mathbb{E}[\tilde{\xi}_{i,j}] &= a_{i,j} + \sum_{\ell=2}^{K+K'} a_{i, \ell} \mathbb{E}[\tilde{\xi}_{\ell,j}].
\end{align}
Using the notation from the proof of Lemma~\ref{lem:aa}, we let
\begin{align*}
	A^* = 
	\begin{pmatrix}
		M^* & \rvline & S \\
		\hline
		0 & \rvline & M'
	\end{pmatrix}
\end{align*}
denote the mean matrix of the subcritical branching process considered there. Furthermore, we let \[
E = (\mathbb{E}[\tilde{\xi}_{i,j}])_{1 \le i,j \le K+K'}.
\]
This way, Equation~\eqref{eq:f1} may be rephrased as
\begin{align*}
E = A + A^* E
\end{align*}
Since $\rho(A^*)<1$ we have that $I - A^*$ is invertible. 
 Hence
\begin{align}
	E = (I - A^*)^{-1} A.
\end{align}
We let $e_1^\top = (1, 0, \ldots, 0)$ denote the first $K+K'$ dimensional unit vector. This way, $A = A^* + Ae_1 e_1^\top$ and therefore
\begin{align*}
		e_1^\top (I-A^*)^{-1} A &= 	e_1^\top  (I-A^*)^{-1}(A^* + Ae_1 e_1^\top ) \\
		&= e_1^\top  (I-A^*)^{-1}(I   + Ae_1 e_1^\top  + (A^* -I)) \\
		&= e_1^\top  (I-A^*)^{-1} +  e_1^\top  (I-A^*)^{-1}A e_1 e_1^\top  - e_1^\top  \\
		&=   e_1^\top  (I-A^*)^{-1}  +  (e_1^\top E e_1 - 1) e_1^\top  \\
		&= e_1^\top  (I-A^*)^{-1},
\end{align*}
because $e_1^\top E e_1 = \mathbb{E}[\tilde{\xi}_{1,1}] = \mathbb{E}[\tilde{\xi}_{1}] = 1$ by Lemma~\ref{lem:aa}. Therefore, $e_1^\top  (I-A^*)^{-1}$ is a left $1$-eigenvector of $A$, and it is equal to the first row of $E$. In other words,
\[
e_1^\top  E = (\mathbb{E}[\tilde{\xi}_{1}], \ldots, \mathbb{E}[\tilde{\xi}_{K+K'}])
\]
is a left $1$-eigenvector of $A$ whose first coordinate is equal to $1$. Since the first $K$ coordinates form necessarily a $1$-eigenvector of $M$, it follows that 
\begin{align}
	\label{eq:important}
	(\mathbb{E}[\tilde{\xi}_{1}], \ldots, \mathbb{E}[\tilde{\xi}_{K+K'}]) = \frac{1}{a_1} (\bm{a}, \bm{a}').
\end{align}
Thus, we arrive at
\begin{align*}
	c_{\mathrm{scal}}' = 	\frac{\sigma}{2} \sqrt{\sum_{i=1}^{K+K'} \lambda_i a_i}.
\end{align*}
This completes the proof.

\section{Conditioning by types}\label{sec:types}

We observe that for the regime under consideration we condition on an event whose probability decays polynomially. A similar result was given by P\'enisson~\cite{MR3476213}, Proposition 5.4.

\begin{lemma}
	\label{le:condpoly}
	Under the conditions of Theorem~\ref{thm:bytype}, there exists $s>0$ with
	\[
				\Prb{(\#_i \cT)_{i \in I} = x(n)} = \Theta\left( \frac{1}{n^{s}}\right).
	\]
\end{lemma}
\begin{proof}
	Let $\bm{\Sigma}$ denote the covariance matrix of $(\tilde{\xi}_i)_{i \in I\cup\{1\}}$, $\varphi_{\bm{0}, \bm{\Sigma}}$ the density of the normal distribution $\mathcal{N}(\bm{0}, \bm{\Sigma})$, and $m$ as in Proposition~\ref{prop:part} for the random vector $\bm{X} := (\tilde{\xi}_i)_{i \in I\cup\{1\}}$.
	
	As stated in Proposition~\ref{prop:part}, if the smallest lattice containing the support of $\bm{X}$ has full rank, then $m>0$. We are first going to treat this case. 
	
	Furthermore, we make a case distinction on whether $1 \in I$, and we first treat the case $1 \in I$. Let $I^* := I \setminus\{1\}$. Analogously as in Equation~\eqref{eq:eq1} we obtain
	\begin{align}
		\label{eq:cbtstep1}
	\Prb{(\#_i \cT)_{i \in I} = x(n)} &= \frac{1}{x_1(n)} \Prb{\sum_{j=1}^{x_1(n)} \left(\tilde{\xi}_1^{(j)} -1, (\tilde{\xi}_i^{(j)})_{i \in I^*}\right) = \left(-1, (x_i(n))_{i \in I^*}\right)}.
	\end{align}
	Here $(\tilde{\xi}_i^{(j)})_{i \in I}$, $j=1,2, \ldots$ denote independent copies of $(\tilde{\xi}_i)_{i \in I}$. By Equation~\eqref{eq:important} we know that
	\[
		\Ex{\tilde{\xi}_i} = \frac{a_i}{a_1}
	\]
	and by Equation~\eqref{eq:cccond} we may write
	\[
		x_i(n) = a_i n + t_i(n) \sqrt{n}
	\]
	with $t_i(n) = O(1)$  as $n \to \infty$ for all $i \in I$.  By Proposition~\ref{pro:llt}, it follows from Equation~\eqref{eq:cbtstep1} that
	\begin{align}
		\label{eq:with1}
		&\Prb{(\#_i \cT)_{i \in I} = x(n)} \\ 
		&= \frac{1}{x_1(n)} \Prb{\sum_{j=1}^{x_1(n)} \frac{1}{\sqrt{n}} \left(\tilde{\xi}_i^{(j)} -\Ex{\tilde{\xi}_i}\right)_{i \in I}  = \left(-\frac{1}{\sqrt{n}}, \left(t_i(n) - t_1(n) \frac{a_i}{a_1}\right)_{i \in I^*}\right)  } \nonumber \\
		&\sim \frac{m}{(a_1 n)^{1 + |I|/2}} \varphi_{\bm{0}, \bm{\Sigma}}\left(\left(t_i(n) - t_1(n) \frac{a_i}{a_1}\right)_{i \in I}\right). \nonumber
	\end{align}
	 Since $t_i(n) = O(1)$ for all $i \in I$ this simplifies to 
	\[
		\Prb{(\#_i \cT)_{i \in I} = x(n)} = \Theta\left( \frac{1}{n^{1 + |I|/2}}\right).
	\]
	It remains to treat the case $1 \notin I$. For a positive integer $x_1(n) = a_1n + t_1(n) \sqrt{n}$ we obtain analogously as for Equation~\eqref{eq:with1} that
	\begin{align}
		\label{eq:concentrationwith1}
		&\Prb{(\#_i \cT)_{i \in I\cup\{1\}} = (x_i(n))_{i \in I \cup \{1\}}} \\
		&= \frac{m}{(x_1(n))^{3/2 + |I|/2}} \varphi_{\bm{0}, \bm{\Sigma}}\left( \left(t_i(n) - t_1(n) \frac{a_i}{a_1}\right)_{i \in I\cup\{1\}} + \frac{o(1)}{\max(1, \sum_{i \in I} \left(t_i(n) - t_1(n) \frac{a_i}{a_1}\right)^2 )}\right)\nonumber \\
		&= \frac{m}{(x_1(n))^{3/2 + |I|/2}} \varphi_{\bm{0}, \bm{\Sigma}}\left( \left(t_i(n) - t_1(n) \frac{a_i}{a_1}\right)_{i \in I\cup\{1\}} + \frac{o(1)}{\max(1, t_1(n)^2 )}\right).\nonumber
	\end{align}
	  Summing over $x_1(n)$ means summing over $t_1(n)$ in increments of a fixed multiple of $1/\sqrt{n}$, yielding by Riemann summation that
	\[
			\Prb{(\#_i \cT)_{i \in I} = x(n)} = \Theta\left( \frac{1}{n^{1 + |I|/2}}\right).
	\]
	This concludes the case $1 \notin I$.
	
	Hence, the statement of the present lemma holds for $s = 1 + |I|/2$ if the smallest lattice containing the support of $\bm{X}$ has full rank. The case where the rank is strictly less than $|I \cup\{1\}|$ may be treated by a change of coordinates and arguing analogously. This concludes the proof.	
\end{proof}

Equation~\eqref{eq:concentrationwith1} already shows that if $1 \notin I$ then $\#_1(\mT_n)/n \convp a_1$. More strongly, Lemma~\ref{le:condpoly} allows us to argue analogously as in Proposition~\ref{prop:concentration}, showing that if $1 \notin I$ then
there exist constants $A,B>0$ independent of $n$ such that, for all $n \geq 1$: 
\begin{align*}
	\mathbb{P}\left( \left|{\#_1 (\mT_n)}-a_1 n\right| > n^{\frac{1}{2}+\delta} \right) \leq A \exp\left( -B n^\delta \right).
\end{align*}
Of course, in case $1 \in I$ this holds trivially by our assumptions on $x_1(n)$.

Thus, we may argue fully analogously as in the proof of Theorem~\ref{thm:reducible}, obtaining first
\[
	\left(\mT_n^{\redu}, \frac{\sqrt{\mathbb{V}[\tilde{\xi}_1]}}{2} (a_1 n)^{-\frac{1}{2}}d_{\mT_n^{\redu}},m_{\mT_n^{\redu}}\right)\convd (\mathcal{T}_{\be},d_{\mathcal{T}_{\be}},m_{\mathcal{T}_{\be}}),
\]
then the same scaling limit for $\mT_n^{\flat}$, and then by comparison with the size-biased tree,
\[
	\left(\mT_n, \frac{\mathfrak{c}_{\mathrm{scal}}}{\sqrt{n}}d_{\mT_n},m_{\mT_n}\right)\convd (\mathcal{T}_{\be},d_{\mathcal{T}_{\be}},m_{\mathcal{T}_{\be}})
\]
for
\[
	\mathfrak{c}_{\mathrm{scal}} = 	\frac{\sqrt{\mathbb{V}[\tilde{\xi}_1]}}{2 c_{\varsigma} \sqrt{a_1}}.
\]
By Equation~\eqref{eq:redmimi} we have $	\frac{\sqrt{\mathbb{V}[\txi_1] }}{c_{\varsigma}}	  = \sigma \sqrt{a_1}$, hence this simplifies to
\begin{align}
	\mathfrak{c}_{\mathrm{scal}} = \frac{\sigma}{2}.
\end{align}
The tail bound for the height of $\mT_n$ may be obtained fully analogously as in the proof of Theorem~\ref{thm:reducible}. With this, the proof of Theorem~\ref{thm:bytype} is complete.
\begin{appendix}
\section*{\label{sec:app}}
\subsection{Medium deviation inequalities}

We make use of the following medium deviation inequality (see \cite{RS15}, Example 1.4).

\begin{lemma}\label{DeviationLemma}
	Let $(X_i)_{i\in\mathbb{N}}$ be an i.i.d. family of real-valued random variables, satisfying $\mathbb{E}[X_1]=0$ and $\mathbb{E}[e^{sX_1}]<\infty$ for all $s$ in some open interval around zero. Then there exist constants $r,D>0$, such that for all $n\in \mathbb{N}, x\geq 0$ and $0\leq \varrho\leq r$ it holds that
	\begin{align*}
		\Prb{|X_1+\dots+X_n|\geq x}\leq 2 \exp(Dn\varrho^2-\varrho x).
	\end{align*}
\end{lemma}

\begin{proof}
	Let $r>0$ such that $\mathbb E[e^{sX_1}]<\infty$ for all $s$ such that $|s| \leq r$, and set $C:=\mathbb E[e^{r|X_1|}]$. In particular, using the fact that $e^x \geq \frac{x^k}{k!}$ for all $x \geq 0$, we get that, for all $k \geq 0$, $|\mathbb E[X_1^k]|\leq C k! r^{-k}$. In particular, $\mathbb{V}[X_1]<\infty$.
	Thus, for any $s \in (0,r)$, we have
	\begin{align*}
		\mathbb E[e^{s X_1}] &= \sum_{k \geq 0} \frac{s^k}{k!} \mathbb E[X_1^k]\\
		&\leq 1+\frac{s^2}{2} \mathbb{V}[X_1] + C \sum_{k \geq 3} (s/r)^k\\
		&\leq 1+C' s^2
	\end{align*}
	uniformly for $s \in (0,r)$ small enough, for some constant $C'$.
	Using Markov's inequality, we get for any $\varrho \in (0,r)$:
	\begin{align*}
		\Prb{X_1+\ldots+X_n \geq x} &\leq e^{-\varrho x} \mathbb E\left[ e^{\varrho X_1} \right]^n\\
		&\leq e^{-\varrho x} (1+C' \varrho^2)^n\\
		&\leq e^{C' \varrho^2 n - \varrho x}.
	\end{align*}
	We bound $\Prb{X_1+\ldots+X_n \leq -x}$ the same way, which proves the result.
\end{proof}

\subsection{Multi-dimensional local limit theorems}
We recall a multi-dimensional local limit theorem following closely the presentation in~\cite{MR4414401}.

Let  $d \ge 1$ denote an integer. A \emph{lattice in  $\Z^d$} is a subset of the form \[
L_{\bm{g}, \bm{A}} = \{\bm{g} + \bm{A} \bm{x} \mid \bm{x} \in \Z^d\}
\]
for  $\bm{g} \in \Z^d$ and $\bm{A} \in \Z^{d \times d}$. For each vector $\bm{h} \in L_{\bm{g}, \bm{A}}$ we have that $L_{\bm{g}, \bm{A}} = L_{\bm{h}, \bm{A}}$. Hence, if we shift a lattice by the negative of any of its elements we obtain a $\Z$-linear submodule of $\Z^d$.  Note that the matrix $\bm{A}$ is not uniquely determined by the lattice. In fact, any  matrix whose columns form a $\Z$-linear generating set of $\{\bm{A} \bm{x} \mid \bm{x} \in \Z^d\}$  will do.

Given a nonempty subset $\Omega \subset \Z^d$   there exists a smallest lattice in $\Z^d$ containing $\Omega$: Select $\bm{g} \in \Omega$ and let $\Lambda$ denote the $\Z$-span of $\Omega - \bm{g}$. Any $\Z$-submodule of $\Z^d$ has rank at most $d$. It follows that there is at least one matrix $\bm{A} \in \Z^{d \times d}$ with $\Lambda = \{\bm{A} \bm{x} \mid \bm{x} \in \Z^d\}$. Thus $L_{\bm{g}, \bm{A}}$ is a lattice containing $\Omega$. Furthermore, $L_{\bm{g}, \bm{A}}$ is a subset of any lattice containing $\Omega$: If $S = L_{\bm{h}, \bm{B}}$ is another lattice containing $\Omega$, then it follows that $\bm{g} \in S$ and hence $S = L_{\bm{g}, \bm{B}}$. This yields that the $\Z$-module $S - \bm{g}$ must contain $\Omega - \bm{g}$, and hence  $\Lambda$ is a submodule of $S - \bm{g}$. Hence $L_{\bm{g}, \bm{A}} \subset S$. 

Let $\bm{X}$ denote a random vector with values in an abelian group $F \simeq \Z^d$. We call $\bm{X}$ \emph{aperiodic}, if the smallest subgroup $F_0$ of $F$ that contains the support $\mathrm{supp}(\bm{X})$ satisfies $F_0 = F$. This is not actually a restriction, because the structure theorem for finitely generated modules over a principal ideal domain yields that $F_0 \simeq \Z^{d'}$ for some $0 \le d' \le d$. We say the random vector $\bm{X}$ (and the associated random walk with step distribution~$\bm{X}$) is \emph{strongly aperiodic}, if the smallest semi-group (a subset closed under addition that contains $\bm{0}$)  $F_1$ of $F$ that contains $\mathrm{supp}(\bm{X})$ satisfies $F_1 = F$. Note that this, on the other hand, does constitute an actual restriction. For example, if $\bm{X}$ is $2$-dimensional with support $\{ (1,0), (0,1), (1,2) \}$, then it is aperiodic, but not strongly aperiodic (although the support is not even contained on any straight line). 

\begin{proposition}[{\cite{MR0279849}, Proposition 1}]
	\label{prop:part}
	Let $\bm{X}$ be a random vector in $\Z^d$. Let  $\bm{g} + \bm{D} \Z^d$ be the smallest lattice  containing the support of $\bm{X}$. Suppose that  $\bm{g} \in \mathrm{supp}(\bm{X})$ and that $\bm{D}$ has full rank, so that $m := |\det \bm{D}|$ is a positive integer. Let $(\bm{X}_i)_{i \ge 1}$ denote independent copies of $\bm{X}$ and set
	\begin{align}
		\bm{S}_n := \sum_{i=1}^n \bm{X}_i.
	\end{align}
	Then:
	\begin{enumerate}
		\item The support of $\bm{S}_m$ generates $\bm{D} \Z^d$ as additive group.
		\item For all $k \ge 1$ and $1 \le j \le m$ it holds that $\Prb{\bm{S}_{km +j} \in j \bm{g} + \bm{D} \Z^d} = 1$.
	\end{enumerate}
\end{proposition}

The following is a strengthened and generalized multi-dimensional version of Gnedenko's local limit theorem, that applies to the case of lattice distributed random variables that are aperiodic but not necessarily strongly aperiodic.

\begin{proposition}[Strengthened local central limit theorem for lattice distributions]
	\label{pro:llt}
	Let $\bm{X}$, $m$, $\bm{g}$, $\bm{D}$ and $\bm{S}_n$ be as in Proposition~\ref{prop:part}. 
	Suppose that $\bm{X}$ has a finite  covariance matrix $\bm{\Sigma}$. Our assumptions imply that $\bm{\Sigma}$ is positive-definite. Let
	\[
	\varphi_{\bm{0}, \bm{\Sigma}}(\bm{y}) = \frac{1}{\sqrt{(2 \pi)^d \det \bm{\Sigma}}} \exp\left( - \frac{1}{2} \bm{y}^\intercal \bm{\Sigma}^{-1} \bm{y} \right)
	\]
	be the  density of the normal distribution $\mathcal{N}(\bm{0}, \bm{\Sigma})$. Set
	\[
	\bm{A}_n = n \Ex{\bm{X}} \qquad \text{and} \qquad B_n = \sqrt{n},
	\]
	so that $B_n^{-1}(\bm{S}_n - \bm{A}_n) \convd \mathcal{N}(\bm{0}, \bm{\Sigma})$. Set 	\[
	R_n(\bm{x}) = \max\left(1, {  \| B_n^{-1} (\bm{x} - \bm{A}_n) \|_2^2} \right).
	\]
	Then for each integer $1 \le j \le m$ 
	\begin{align}
		\label{eq:tof}
		\lim_{\substack{n \to \infty \\ n \in j + m \Z}} \sup_{\bm{x} \in j\bm{g} + \bm{D} \Z^d} R_n(\bm{x})  |B_n^d  \Prb{\bm{S}_n = \bm{x}} - m \varphi_{\bm{0},\bm{\Sigma}}\left(B_n^{-1}(\bm{x} - \bm{A}_n) \right)| = 0.
	\end{align}
\end{proposition}
	Gnedenko's one-dimensional local limit theorem of~\cite{MR0026275} has first been generalised in~\cite{rva} to the lattice case of strongly aperiodic multi-dimensional random walk. A further generalisation can be found in~\cite{MR0279849}, Proposition 2, so that strong aperiodicity of $\bm{X}$ is not assumed, but only that it is aperiodic. The version stated of the local limit theorem here in Proposition~\ref{pro:llt} is strengthened by the factor $R_n(\bm{x}) $ in Equation~\eqref{eq:tof}. When $\bm{x}$ deviates sufficiently from $\Ex{\bm{S}_n}$, this is indeed stronger than the original statement. For the strongly aperiodic case such a strengthening was obtained in~\cite{MR0388547}, Statement P10, page 79 via a modification of the proof. The version here in Proposition~\ref{pro:llt} may be derived analogously as in~\cite{MR0279849} from the strongly aperiodic setting in~\cite{MR0388547}.
\end{appendix}
%
%

\begin{acks}[Acknowledgments]
We warmly thank the referees for their helpful comments and thorough reading of the manuscript.
\end{acks}
\begin{funding}
This research was funded in part by the Austrian Science Fund (FWF) 10.55776/F1002 and 10.55776/PAT6732623 and MSCA-RISE-2020 project RandNET, Contract 101007705. LAB acknowledges support of NSERC and the Canada Research Chairs program. The computational results have been achieved using the Austrian Scientific Computing (ASC) infrastructure. For open access purposes, the authors have applied a CC BY public copyright license to any author-accepted manuscript version arising from this submission.
\end{funding}



\bibliographystyle{imsart-number} 
\bibliography{lib.bib}       


\end{document}